\documentclass[a4paper,12pt]{amsart}

\usepackage{float}
\usepackage{euscript,eufrak,verbatim}
\usepackage{graphicx}
\usepackage[usenames]{color}
\usepackage[colorlinks,linkcolor=red,anchorcolor=blue,citecolor=blue]{hyperref}
\usepackage{amsmath}
\usepackage{amsthm}
\usepackage[all]{xy}
\usepackage{graphicx} 
\usepackage{amssymb} 
\usepackage{amsfonts}

\usepackage{mathrsfs}
\usepackage{amscd}

\makeatletter
\newcommand{\svdots}{%
  \vbox{\fontsize{\sf@size}{\sf@size pt}\linespread{0.3}\selectfont
    \kern0.2\baselineskip
    \hbox{.}\hbox{.}\hbox{.}%
    \kern0.1\baselineskip
  }%
}
\makeatother

\theoremstyle{plain}
\newtheorem{main theorem}{Main Theorem}
\newtheorem{theorem}{Theorem}[section]
\newtheorem{lemma}[theorem]{Lemma}

\newtheorem{corollary}[theorem]{Corollary}
\newtheorem{proposition}[theorem]{Proposition}
\newtheorem{claim}[theorem]{Claim}

\newtheorem{lemma-definition}[theorem]{Lemma-Definition}
\theoremstyle{definition}

\newtheorem{remark}[theorem]{Remark}

\newtheorem{problem}[theorem]{Problem}

\makeatletter 
\@addtoreset{equation}{section}
\numberwithin{equation}{section}

\newcommand{\diam}{\mathrm{Diam}}
\newcommand{\supp}{\mathrm{supp}}
\newcommand{\lip}{\mathrm{Lip}}
\newcommand{\fix}{\mathrm{Fix}}
\newcommand{\gr}{\mathrm{Gr}}

\title[Multidimensional Lipschitz Bebutov--Kakutani theorem]
{A Lipschitz Refinement of the Multidimensional Bebutov--Kakutani Dynamical Embedding Theorem}

\author{Yonatan Gutman, Qiang Huo, Masaki Tsukamoto}

\address[Yonatan Gutman]
{Institute of Mathematics, Polish Academy of Sciences, ul. \'{S}niadeckich 8, Warsaw, 00-656, Poland.}

\email{gutman@impan.pl}

\address[Qiang Huo]
{School of Mathematical Sciences, University of Science and Technology of China, Hefei 230026, China.}

\email{qianghuo@ustc.edu.cn}

\address[Masaki Tsukamoto]
{Department of Mathematics, Kyoto University, Kyoto 606-8502, Japan.}

\email{tsukamoto@math.kyoto-u.ac.jp}

\begin{document}



\thanks{Y.G. was partially supported by the National Science Centre (Poland) grant 2020/39/B/ST1/02329. Q.H. was partially supported by China Postdoctoral Science Foundation 2025M773065; National Key Research and Development Program of China 2024YFA1013600; Fundamental Research Funds for the Central Universities. 
M.T. was supported by JSPS KAKENHI JP25K06974.}

\maketitle

\begin{abstract}
We prove that a continuous action of $\mathbb{R}^n$ on a compact metrizable space equivariantly embeds into 
the shift action on the space of one-Lipschitz functions from $\mathbb{R}^n$ to $[0,1]$ if and only if
the set of fixed points topologically embeds in $[0,1]$.
This is a Lipschitz refinement of classical dynamical embedding theorems of Bebutov, Kakutani, Jaworski and Chen.
\end{abstract}

\section{Introduction}  \label{section: Introduction}

\subsection{The Bebutov--Kakutani theorem and its multidimensional generalization} \label{subsection: Bebutov--Kakutani theorem}

The purpose of this paper is to develop a Lipschitz refinement of the multidimensional Bebutov--Kakutani dynamical embedding theorem.
We first review the classical theory.
We will present our main result in \S \ref{subsection: main result} below.

Throughout this paper we assume that $\mathbb{R}^n$ is endowed with the Euclidean topology and the standard additive group structure.
Let $I = [0,1] = \{t\in \mathbb{R}\mid 0\leq t \leq 1\}$ be the unit interval. 
We define $C(\mathbb{R}, I)$ as the space of continuous maps $\varphi\colon \mathbb{R}\to I$ endowed with the compact-open topology.
The group $\mathbb{R}$ naturally acts on it by the translation: 
\[  \mathbb{R}\times C(\mathbb{R}, I) \to C(\mathbb{R}, I), \quad  \left(a, \varphi(t)\right) \mapsto  \varphi(t+a). \]
The set of fixed points of this action consists of constant maps and is homeomorphic to the unit interval $I$.

Let $X$ be a compact metrizable space and $T\colon \mathbb{R}\times X\to X$ a continuous action of $\mathbb{R}$ on it.
Denote $\fix (X, T) = \{x\in X\mid T^t x = x \, (\forall t\in \mathbb{R})\}$.
We say that \textbf{$(X, T)$ equivariantly embeds in $C(\mathbb{R}, I)$} if there exists a topological embedding 
$f\colon X\to C(\mathbb{R}, I)$ that commutes with the $\mathbb{R}$-actions.

Bebutov \cite{Bebutov} and Kakutani \cite{Kakutani} found that $C(\mathbb{R}, I)$ is a so-called \textit{universal embedding space}. 

\begin{theorem}[Bebutov 1940, Kakutani 1968] \label{theorem: Bebutov--Kakutani}
Let $T\colon \mathbb{R}\times X\to X$ be a continuous action of $\mathbb{R}$ on a compact metrizable space $X$.
It equivariantly embeds in $C(\mathbb{R},I)$ if and only if $\fix(X, T)$ topologically embeds in the unit interval $I$.
\end{theorem}

The “only if” part of the statement is trivial because the set of fixed points of $C(\mathbb{R}, I)$ is homeomorphic to $I$.
The “if” part is highly nontrivial and surprising.

We are interested in the generalization of Theorem \ref{theorem: Bebutov--Kakutani} to $\mathbb{R}^n$-actions, 
which was accomplished by Jaworski \cite{Jaworski} and Chen \cite{Chen75a, Chen75b}.
Indeed they obtained a generalization to much more general topological groups, including all connected locally compact groups.
However here we concentrate on the case of $\mathbb{R}^n$-actions.

Let $C(\mathbb{R}^n, I)$ be the space of continuous maps $\varphi\colon \mathbb{R}^n\to I$ equipped with the compact 
open topology. The group $\mathbb{R}^n$ continuously acts on it by the translation.
Let $T\colon \mathbb{R}^n\times X\to X$ be a continuous action of $\mathbb{R}^n$ on a compact metrizable space $X$.
We set $\fix(X, T) = \{x\in X\mid T^t x  = x \, (\forall t\in \mathbb{R}^n)\}$.
We say that $(X, T)$ equivariantly embeds in $C(\mathbb{R}^n, I)$ if there exists a topological embedding 
$f\colon X\to C(\mathbb{R}^n, I)$ that commutes with the $\mathbb{R}^n$-actions. 

The following theorem is a special case of the results of Jaworski \cite{Jaworski} and Chen \cite{Chen75a, Chen75b}.

\begin{theorem}[Jaworski 1974, Chen 1975]  \label{theorem: multidimensional Bebutov--Kakutani}
Let $T\colon \mathbb{R}^n\times X\to X$ be a continuous action of $\mathbb{R}^n$ on a compact metrizable space $X$.
It equivariantly embeds in $C(\mathbb{R}^n, I)$ if and only if $\fix(X, T)$ topologically embeds in the unit interval $I$.
\end{theorem}

We briefly explain the idea of the proof of this theorem.
As in Theorem \ref{theorem: Bebutov--Kakutani}, the “only if” part of the statement is trivial.
We are going to outline the proof of the “if” part.
So we assume that we are given a topological embedding $\iota \colon \fix(X, T)\to I$.
By the Tietze extension theorem, we can find an equivariant continuous map (not necessarily embedding) 
$f\colon X\to C(\mathbb{R}^n, I)$ that satisfies $f(x)(t) = \iota(x)$ for $x\in \fix(X, T)$ and $t\in \mathbb{R}^n$.
(Here $f(x)(t)\in I$ is the value of the continuous function $f(x)\in C(\mathbb{R}^n, I)$ at the point $t\in \mathbb{R}^n$.)
We try to apply successive perturbation to $f$ over $X\setminus \fix(X, T)$ so that the resulting map becomes an embedding.

Suppose $p$ and $q$ are two distinct points of $X$.
Assume $p, q\not\in \fix(X, T)$.
Then there is $u\in \mathbb{R}^n$ with $T^u p \neq p$.
Let $L = \mathbb{R}u\subset \mathbb{R}^n$ be the one-dimensional subspace generated by $u$.
We consider the restriction $T|_L$ of $T$ to the subgroup $L$.
Since $L$ is one-dimensional, we may apply \cite[Lemma 2]{Kakutani}---used in the proof of Theorem \ref{theorem: Bebutov--Kakutani}---to $T|_L$, thereby obtaining small neighborhoods $A$ and $B$ of $p$ and $q$ respectively and an $L$-equivariant continuous map 
$g\colon X\to C(L, I)$ such that 
\begin{itemize}
  \item $g(x)$ is a small perturbation of $f(x)|_L\colon L\to I$,
  \item $g(A) \cap g(B) = \emptyset$.
\end{itemize}

We define an $\mathbb{R}^n$-equivariant continuous map $h\colon X\to C(\mathbb{R}^n, I)$ by 
$h(x)(t) = g(T^t x)(0)$ $(x\in X, t\in \mathbb{R}^n)$.
Then we have $h(x) \approx f(x)$ and that $h(A) \cap h(B) = \emptyset$.
Namely, $h$ is a small perturbation of $f$ that can distinguish some neighborhoods of $p$ and $q$.
We iteratively apply this perturbation procedure to $f$ infinitely many times. (This can be formalized by using the 
Baire category theorem.)
Then the resulting map provides an embedding of $X$ into $C(\mathbb{R}^n, I)$.

In summary, the idea is to reduce the problem to the one-dimensional case and apply the proof of Theorem \ref{theorem: Bebutov--Kakutani}.
The key fact is that $\mathbb{R}^n$ contains a lot of one-parameter subgroups.

We have not used the commutativity of $\mathbb{R}^n$ at all. So the argument equally works well for much more general topological groups
\cite{Jaworski, Chen75a, Chen75b}.
For example, Jaworski \cite{Jaworski} proved a generalization of the Bebutov--Kakutani theorem to all connected locally compact groups.
(The paper \cite{Jaworski} is his thesis at the university of Maryland and may be not very easily accessible.
However readers can read an outline of the proof in the book of Auslander \cite[pp.189-190]{Auslander}.)

\begin{remark}
A key fact, once again, is that every connected locally compact group contains a a rich supply of one-parameter subgroups.
The opposite extreme is the case of discrete groups (say, the group of integers $\mathbb{Z}$).
Discrete groups have no one-parameter subgroups.
Therefore the embedding problem (for e.g. $C(\mathbb{Z}, I)$) has a completely different nature.
The first step in this problem was given in the paper of Jaworski \cite{Jaworski}.
The next decisive step was given by 
Lindenstrauss and Weiss \cite{Lindenstrauss--Weiss, Lindenstrauss}, which discovered that 
\textit{mean dimension theory} \cite{Gromov99} plays a crucial role in the problem.
Readers may consult
\cite{Gutman--Tsukamoto embedding, Gutman--Qiao--Tsukamoto} for recent developments in this direction.
\end{remark}

\subsection{Lipschitz refinements and the main result} \label{subsection: main result}

The Bebutov--Kakutani theorem (Theorem \ref{theorem: Bebutov--Kakutani}) is an excellent result, however 
there is an unsatisfactory point.
In the theorem, we are given an $\mathbb{R}$-action on a \textit{compact} space $X$, but the target space $C(\mathbb{R}, I)$ is 
not compact (nor locally compact). 
To overcome this point, Jin and the first and third named authors of the present paper developed a 
\textit{Lipschitz refinement} of the theorem in \cite{GJT}.

Let $\lip_1(\mathbb{R}, I)$ be the space of maps $\varphi\colon \mathbb{R} \to I$ that satisfy 
$|\varphi(t) - \varphi(u)| \leq |t-u|$ for all $t, u\in \mathbb{R}$.
This is a compact subset of $C(\mathbb{R}, I)$.
The compactness follows from the Arzela--Ascoli theorem.
The group $\mathbb{R}$ naturally acts on it by the translation as before.

The paper \cite{GJT} proved:

\begin{theorem}[Jin--Gutman--Tsukamoto 2019]  \label{theorem: Lipschitz Bebutov--Kakutani}
Let $T\colon \mathbb{R}\times X\to X$ be a continuous action of $\mathbb{R}$ on a compact metrizable space $X$.
It equivariantly embeds in $\lip_1(\mathbb{R},I)$ if and only if $\fix(X, T)$ topologically embeds in the unit interval $I$.
\end{theorem}

The space $\lip_1(\mathbb{R},I)$ is compact (and metrizable).
Therefore, arguably, this statement is more satisfactory.

As we have already seen in \S \ref{subsection: Bebutov--Kakutani theorem}, 
the Bebutov--Kakutani theorem can be generalized to $\mathbb{R}^n$-actions rather directly, 
but Theorem \ref{theorem: Lipschitz Bebutov--Kakutani} cannot.
The idea of the proof of the multidimensional Bebutov--Kakutani theorem 
(Theorem \ref{theorem: multidimensional Bebutov--Kakutani})
is to reduce the problem to the one-dimensional case.
However the same idea does not work for the Lipschitz version.
The difficulty is that, even if a continuous function $\varphi\colon \mathbb{R}^n\to I$ is Lipschitz on a particular line $L\subset \mathbb{R}^n$,
it does not imply that $\varphi$ is Lipschitz on the entire space $\mathbb{R}^n$.
Therefore we cannot reduce the problem to the one-dimensional case and we need to use the geometry of $\mathbb{R}^n$ more seriously.

Now we have arrived at the point where we formulate our main problem. 
Although it may be a bit tedious, we will carefully introduce our terminology once again.
For $t=(t_1, \dots, t_n)\in \mathbb{R}^n$ we set $|t| = \sqrt{t_1^2+\dots+t_n^2}$.
We define $\lip_1(\mathbb{R}^n, I)$ as the space of 
all one-Lipschitz maps from $\mathbb{R}^n$ to $I = [0,1]$.
Namely, this is the space of all maps $\varphi\colon \mathbb{R}^n\to I$ satisfying 
$|\varphi(t)-\varphi(u)|\leq |t-u|$ for all $t, u\in \mathbb{R}^n$.
This space is compact (with respect to the compact-open topology) and metrizable, e.g. we can define a metric $\mathbf{d}$ 
on it by\footnote{Notice that a sequence $\{\varphi_n\}$ converges to $\varphi$ with respect to this metric if and only if
$\varphi_n$ uniformly converges to $\varphi$ over every compact subset of $\mathbb{R}^n$. So the metric $\mathbf{d}$ defines the 
compact-open topology.} 
\begin{equation} \label{eq: metric on the space of Lipschitz functions}
  \mathbf{d}(\varphi, \psi) = \sup_{k\geq 1} \left(2^{-k} \sup_{|t|\leq k} |\varphi(t)-\psi(t)|\right), \quad 
    \left(\varphi, \psi\in \lip_1(\mathbb{R}^n, I)\right). 
\end{equation}   
The group $\mathbb{R}^n$ continuously acts on $\lip_1(\mathbb{R}^n, I)$ by the shift:
\[ \sigma\colon \mathbb{R}^n\times \lip_1(\mathbb{R}^n, I) \to \lip_1(\mathbb{R}^n, I), \quad 
     \left(u, \varphi(t)\right)\mapsto \varphi(t+u). \]

Let $X$ be a compact metrizable space and $T\colon \mathbb{R}^n\times X\to X$ a continuous action of $\mathbb{R}^n$ on it.
We say that $(X,T)$ \textbf{equivariantly embeds} (or $\mathbb{R}^n$-equivariantly embeds) in $\operatorname{Lip}_1(\mathbb{R}^n,I)$ if there exists a
topological embedding
$
f:X\longrightarrow \operatorname{Lip}_1(\mathbb{R}^n,I)
$
that satisfies
$
f\circ T_t=\sigma_t\circ f$ for all $t\in\mathbb{R}^n.
$
Now we can state our main problem.

\begin{problem} \label{problem: main problem}
State a necessary and sufficient condition for embedding of an  
arbitrarily given $\mathbb{R}^n$-action $(X,T)$ into $\operatorname{Lip}_1(\mathbb{R}^n,I)$.
\end{problem}

The first approach to this problem was given by the first and second named authors in \cite{GH}.
For explaining their result, we need to introduce one more terminology.
Let $T\colon X\times \mathbb{R}^n\to X$ be a continuous action of $\mathbb{R}^n$ on a compact metrizable space $X$.
We say that $(X, T)$ is \textbf{weakly locally free} if, for any $x\in X\setminus \fix(X, T)$, there is an open neighborhood $U$ of the origin 
in $\mathbb{R}^n$ such that we have $T^t x\neq x$ for all $t\in U\setminus \{0\}$.
The paper \cite{GH} proved the following partial answer to Problem \ref{problem: main problem}.

\begin{theorem}[Gutman--Huo] \label{theorem: weakly locally free}
Let $T\colon \mathbb{R}^n \times X \to X$ be a weakly locally free continuous action of $\mathbb{R}^n$ on a compact metrizable space $X$.
It equivariantly embeds in $\lip_1(\mathbb{R}^n, I)$ if and only if $\fix(X, T)$ topologically embeds in the unit interval $I$.
\end{theorem}

This theorem gives a necessary and sufficient condition for the embeddability under the additional assumption that 
$(X, T)$ is weakly locally free.
The paper \cite{GH} applied Theorem \ref{theorem: weakly locally free} to a problem in ergodic theory and proved 
that the shift action on $\lip_1(\mathbb{R}^n, I)$ becomes a 
topological model for every free measure-preserving $\mathbb{R}^n$-action.
This is the multidimensional version of a classical theorem of Eberlein \cite{Eberlein73}.

Although Theorem \ref{theorem: weakly locally free} is sufficient for the application to an ergodic problem studied in
\cite{GH}, it does not provide a complete answer to Problem \ref{problem: main problem}.
Therefore \cite[Remark 4.1]{GH} asked whether one can remove the assumption that $(X, T)$ is weakly locally free. The
purpose of the present paper is to solve this question affirmatively. Let us state our
main result.
\begin{theorem}[Main theorem] \label{theorem: main theorem}
Let $T\colon \mathbb{R}^n\times X\to X$ be a continuous action of $\mathbb{R}^n$ on a compact metrizable space $X$.
It equivariantly embeds in $\lip_1(\mathbb{R}^n, I)$ if and only if 
$\fix(X, T)$ topologically embeds in the unit interval $I$.
\end{theorem}

Thus we can remove the assumption of weak local freeness in Theorem \ref{theorem: weakly locally free}.
This is a complete answer to Problem \ref{problem: main problem}.

As we have already mentioned, the Bebutov--Kakutani theorem (Theorem \ref{theorem: Bebutov--Kakutani}) can be 
generalized to much more general topological groups, including all connected locally compact groups.
Therefore it is also natural to expect the validity of some noncommutative generalizations of Theorem \ref{theorem: main theorem}.
Probably posing the following question is the first step toward this direction.

\begin{problem}[Open problem]
Can one generalize Theorem \ref{theorem: main theorem} to continuous actions of connected Lie groups?
\end{problem}

Hopefully this problem will be the subject of future work.

\subsection{Strategy of the proof}  \label{subsection: strategy of the proof}

Here we explain the outline of the proof of Theorem \ref{theorem: main theorem}.
Let $T\colon \mathbb{R}^n\times X\to X$ be a continuous action of $\mathbb{R}^n$ on a compact metrizable space $X$.
If it equivariantly embeds in $\lip_1(\mathbb{R}^n, I)$ then the fixed points set $\fix(X, T)$ topologically embeds in 
$\fix\left(\lip_1(\mathbb{R}^n, I), \sigma\right)$. The latter is homemorphic to the unit interval $I = [0,1]$.
Therefore the “only if” part of Theorem \ref{theorem: main theorem} is obvious. 
The problem is to prove the “if” part of the statement.

Suppose $\fix(X, T)$ topologically embeds in $I$.
This means that there exists a topological embedding $\iota\colon \fix(X,T)\to \fix\left(\lip_1(\mathbb{R}^n, I), \sigma\right)$.
Let $\mathcal{C}$ be the space of continuous maps $f\colon X\to \lip_1(\mathbb{R}^n, I)$ that satisfies 
$f\circ T^t = \sigma^t \circ f$ for all $t\in\mathbb{R}^n$ and $f(x) = \iota(x)$ for all $x\in \fix(X, T)$.
We define a metric $D$ on it by 
\[ D(f, g) = \sup_{x\in X}\mathbf{d}(f(x), g(x)), \]
where $\mathbf{d}$ is the metric on $\lip_1(\mathbb{R}^n, I)$ defined in (\ref{eq: metric on the space of Lipschitz functions}).
It is easy to see that $\mathcal{C}$ is nonempty (Lemma \ref{lemma: C is not empty}) and a complete metric space.
We want to show that the set 
\[ \{f\in \mathcal{C}\mid \text{$f$ is a topological embedding}\} \]
contains a dense $G_\delta$-subset by utilizing the Baire category theorem.
This “Baire category approach” to the proof of Theorem \ref{theorem: main theorem}
is the same as in the case of $n=1$.
However the multidimensional case (i.e. the case of $n>1$) is more involved and exhibits a new difficulty.

For $x\in X$ we denote by $\mathfrak{g}_x$ the largest linear subspace $\mathfrak{g}\subset \mathbb{R}^n$ for which 
we have $T^t x = x$ for all $t\in \mathfrak{g}$.
In other words, $\mathfrak{g}_x$ is the connected component the of \textbf{stabilizer of $x$}, $\{t\in \mathbb{R}^n\mid T^t x = x\}$, containing the origin.
For $0\leq k \leq n$ we define 
$X_k = \{x\in X\mid \dim \mathfrak{g}_x\geq k\}$.
We have 
\begin{equation} \label{eq: filtration of X}
 X = X_0 \supset X_1 \supset X_2 \supset \dots \supset X_n = \fix(X, T).
\end{equation}
We also define $X_{n+1} = \emptyset$.

Set $Y = \lip_1(\mathbb{R}^n, I)$ and define 
\[  Y_k = \left\{\varphi \in Y\middle|\, 
\text{\parbox{12cm}{$\exists \mathfrak{g}\subset \mathbb{R}^n$, a linear subspace of dimension $\geq k$ such that  
$\varphi(t+u) = \varphi(t)$ for all $t\in \mathbb{R}^n$ and $u\in \mathfrak{g}.$}}\right\}. \]
Observe that 
$  Y_k \;=\; \bigl(\operatorname{Lip}_1(\mathbb{R}^n, I)\bigr)_k$,
as introduced in the previous paragraph in the context of a general $\mathbb{R}^n$-action on $X$. Thus we have $Y=Y_0\supset Y_1\supset Y_2\supset \dots \supset Y_n = \fix(Y, \sigma)$.
If $f\colon X\to Y$ is an equivariant continuous map then $f(X_k)\subset Y_k$ for all 
$0\leq k \leq n$. 
Therefore we need to respect the filtration structure (\ref{eq: filtration of X})
in all the constructions below.
In the case of $n=1$, this is not difficult because we have only two terms 
$X \supset \fix(X, T)$.
But it becomes more involved in the case of $n>1$.
(Notice also that we have only two terms $X \supset \fix(X, T)$ if we assume that 
$(X, T)$ is weakly locally free. Therefore the weak local freeness assumption makes the situation closer to the 
one-dimensional case.)

Let $0\leq k \leq n-1$.
For a closed subset $A\subset X_k$ 
with $A\cap X_{k+1}=\emptyset$ we set
\[  \mathcal{C}_k(A) = \{f\in \mathcal{C}\mid f(A) \cap Y_{k+1} = \emptyset\} .\]
We define $\tilde{X}_k = \{(x, y)\in (X_k\setminus X_{k+1})\times (X_k\setminus X_{k+1})\mid 
x\neq y \text{ and } \mathfrak{g}_x=\mathfrak{g}_y\}$.
For closed subsets $B, B^\prime\subset X_k$ with 
$B\cap B^\prime = B\cap X_{k+1} = B^\prime\cap X_{k+1}=\emptyset$, we define 
\[ \mathcal{C}_k(B, B^\prime) = 
   \{f\in \mathcal{C}\mid \forall (x,y) \in (B\times B^\prime)\cap \tilde{X}_k: f(x)\neq f(y)\}. \]
It is easy to see that the above $\mathcal{C}_k(A)$ and $\mathcal{C}_k(B, B^\prime)$
are open sets of $\mathcal{C}$.

The following two propositions are main steps towards the proof of Theorem 
\ref{theorem: main theorem}.

\begin{proposition} \label{proposition: C_k(A) is open and dense}
Let $0\leq k \leq n-1$ and $p\in X_k\setminus X_{k+1}$.
There exists a closed neighborhood $A$ of $p$ in $X_k$ with $A\cap X_{k+1}=\emptyset$
such that $\mathcal{C}_k(A)$ is open and dense in $\mathcal{C}$.
\end{proposition}

\begin{proposition} \label{proposition: C_k(B, B^prime) is open and dense}
Let $0\leq k \leq n-1$ and $(p,q) \in \tilde{X}_k$.
There exist closed neighborhoods $B$ and $B^\prime$ of $p$ and $q$ in $X_k$ respectively 
such that $B \cap X_{k+1}= B^\prime \cap X_{k+1} = B\cap B^\prime =\emptyset$
and that $\mathcal{C}_k(B,B^\prime)$ is open and dense in $\mathcal{C}$.
\end{proposition}

Assuming these propositions (and the fact $\mathcal{C}\neq \emptyset$; Lemma \ref{lemma: C is not empty}),
we can prove Theorem \ref{theorem: main theorem} as follows.

\begin{proof}[Proof of Theorem \ref{theorem: main theorem}, 
assuming Propositions \ref{proposition: C_k(A) is open and dense} 
and \ref{proposition: C_k(B, B^prime) is open and dense}]
As we already remarked, it is enough to prove that,
given a topological embedding $\iota\colon X_n \to Y_n$, 
the set $\mathcal{C}$ defined above contains an embedding.

We notice that every separable metric space is a Lindel\"{o}f space (i.e. every open cover has a countable subcover).

Let $0\leq k \leq n-1$.
By Proposition \ref{proposition: C_k(A) is open and dense} and the Lindel\"{o}f property of $X_k\setminus X_{k+1}$, 
there exist closed subsets $A_{kj}$ of $X_k$
$(j\geq 1)$ such that $A_{kj}\cap X_{k+1}=\emptyset$, $X_k\setminus X_{k+1} = \bigcup_{j=1}^\infty A_{kj}$ and that 
$\mathcal{C}_k(A_{kj})$ is open and dense in $\mathcal{C}$.

By Proposition \ref{proposition: C_k(B, B^prime) is open and dense} and the Lindel\"{o}f property of $\tilde{X}_k$, 
there exist closed subsets $B_{kj}$ and $B^\prime_{kj}$ of $X_k$ 
$(j\geq 1)$ such that 
$B_{kj}\cap X_{k+1} = B^\prime_{kj}\cap X_{k+1} = B_{kj}\cap B^\prime_{kj} = \emptyset$, 
$\tilde{X}_k \subset \bigcup_{j=1}^\infty \left(B_{kj} \times B^\prime_{kj}\right)$ and that 
$\mathcal{C}_k(B_{kj}, B^\prime_{kj})$ is open and dense in $\mathcal{C}$. 

Since $\mathcal{C}$ is a (nonempty) complete metric space, the Baire category theorem implies that the set 
\[  \bigcap_{k=0}^{n-1} \bigcap_{j=1}^\infty \left(\mathcal{C}_k(A_{kj})\cap \mathcal{C}_k(B_{kj}, B^\prime_{kj})\right) \]
is dense in $\mathcal{C}$ (and hence nonempty).
Take a map $f$ in this set. We are going to prove that $f$ is an embedding.

Let $p, q \in X$ be any distinct two points.
We want to show $f(p) \neq f(q)$.
Suppose $f(p) = f(q)$.
If $p\in X_k\setminus X_{k+1}$ and $q\in X_{\ell}$ with $k<\ell$, then we can find $A_{kj}$ that contains $p$ and hence 
$f(p) \not\in Y_{k+1}$ since $f\in \mathcal{C}_k(A_{kj})$.
However we have $f(q) \in Y_\ell \subset Y_{k+1}$. Thus $f(p)\neq f(q)$.

Therefore we can assume that both $p$ and $q$ belong to the same $X_k\setminus X_{k+1}$ $(0\leq k \leq n)$.
If $k=n$ then $f(p) = \iota(p) \neq \iota(q) =f(q)$. Hence we assume $k\leq n-1$.

If $\mathfrak{g}_p \neq \mathfrak{g}_q$ then the map $f(p) = f(q)$ is invariant under $\mathfrak{g}_p + \mathfrak{g}_q$ and hence 
belongs to $Y_{k+1}$.
This again contradicts the fact that $p\in A_{kj}$ for some $j$ and $f\in \mathcal{C}_k(A_{kj})$.
Therefore we have $\mathfrak{g}_p = \mathfrak{g}_q$, namely $(p,q) \in \tilde{X}_k$.
Then $(p, q)\in B_{kj}\times B^\prime_{kj}$ for some $j$.
Since $f\in \mathcal{C}_k(B_{kj}, B^\prime_{kj})$, we have $f(p) \neq f(q)$.
\end{proof}

Next we explain the ideas of the proofs of Propositions \ref{proposition: C_k(A) is open and dense} and 
\ref{proposition: C_k(B, B^prime) is open and dense}.
Their proofs are similar. So we concentrate on the proof of Proposition \ref{proposition: C_k(A) is open and dense}.

The set $\mathcal{C}_k(A)$ is open for any closed subset $A\subset X_k$.
The problem is to prove that it is dense for appropriately chosen $A\subset X_k$.
For each $x\in X$, let $\mathfrak{h}_x \subset \mathbb{R}^n$ be the orthogonal complement of $\mathfrak{g}_x$
in $\mathbb{R}^n$.
We have the orthogonal decomposition $\mathbb{R}^n = \mathfrak{g}_x\oplus \mathfrak{h}_x$.
Set $\mathfrak{h} = \{(x, t) \in X\times \mathbb{R}^n\mid t \in \mathfrak{h}_x\}$.
Over each $X_k\setminus X_{k+1}$, $\mathfrak{h}$ is a vector bundle of rank $n-k$.
(That is, the space $\mathfrak{h}_x$ depends continuously on $x\in X_k\setminus X_{k+1}$.)
For $E\subset X_k\setminus X_{k+1}$ we denote 
$\mathfrak{h}|_E = \{(x, t) \in \mathfrak{h}\mid  x \in E\}$.
For $r>0$, we set $B_r(\mathfrak{h}|_E) = \{(x, t) \in \mathfrak{h}|_E\mid |t|\leq r\}$.

Let $p\in X_k\setminus X_{k+1}$ $(0\leq k\leq n-1)$.
We can prove that there exist a small number $r>0$ and a (small) closed subset $E$ of $X_k$ with 
$p\in E$ and $E\cap X_{k+1} = \emptyset$ such that the map 
\begin{equation} \label{eq: local section map}
  B_r(\mathfrak{h}|_E) \to X_k, \quad (x, t) \mapsto T^t x
\end{equation}
is injective and that its image contains an open neighborhood of $p$ in $X_k$.
We call such a set $E$ \textbf{local section}\footnote{Note that this does not coincide with the notion of a \textit{local section} introduced in \cite[Definition 2.8]{GH}. Indeed, although for $x \in E$ one has $\mathfrak{h}_x \cong \mathbb{R}^{\,n-k}$, in general there is no natural way to induce an associated (local) $\mathbb{R}^{\,n-k}$-action on $X_k$.
}.
Let $A$ be the image of the map (\ref{eq: local section map}).
We want to show that $\mathcal{C}_k(A)$ is dense in $\mathcal{C}$.
(Indeed, in the real proof, we have to define $A$ more carefully. But here we omit the details.) 

Let $f\in \mathcal{C}$ and $\varepsilon>0$ be arbitrary.
We try to perturb $f$ so that the resulting map belongs to $\mathcal{C}_k(A)$.
For $x\in X_k$, let $H(x)$ be the set of $t\in \mathfrak{h}_x$ for which we have $T^t x\in E$.\\
We perturb $f(x)$ locally near each $t \in H(x)$. (This procedure is somewhat analogous to the use of Rokhlin towers in ergodic theory and guarantees that the perturbation of the function $f$ is $\mathbb{R}^n$-equivariant - see heuristic discussion as part of the proof of \cite[Theorem 4.4]{GH}). Since membership in $Y_{k+1}$ is a global condition, it can be invalidated by suitable local perturbations and so  we can construct  
$
g : X_k \to \operatorname{Lip}_1(\mathbb{R}^n, I)
$ 
such that 
$
g(A) \cap Y_{k+1} = \emptyset
\,\,\text{and}\,\,
\sup_{x \in X_k} \mathbf{d}(f(x), g(x)) < \varepsilon.
$\\

Now a difficulty arises: The map $g$ has been defined on the subspace $X_k$ so far.
The problem is how to extend $g$ to the entire space $X$.
By using the Tietze extension theorem, it is easy to extend $g$ to an equivariant continuous map 
$g\colon X\to C(\mathbb{R}^n, I)$, where $C(\mathbb{R}^n, I)$ denotes the space of continuous maps from 
$\mathbb{R}^n$ to $I$.
However this is not sufficient for our purpose here.
We need a map $g\colon X\to \lip_1(\mathbb{R}^n, I)$.
This is a new difficulty which is completely absent in the case of $n=1$
(and the case of weakly locally free $(X, T)$).
We will overcome it by developing a certain “filter” which converts continuous functions into Lipschitz ones.
By applying this filter to $g\colon X\to C(\mathbb{R}^n, I)$, 
we get a map $h \in \mathcal{C}$ that satisfies $D(f, h)< \varepsilon$ and $h(A)\cap Y_{k+1} = \emptyset$.

This concludes the outline of the proof of Proposition \ref{proposition: C_k(A) is open and dense}.
The proof of Proposition \ref{proposition: C_k(B, B^prime) is open and dense} is similar.

\section{Lipschitz filter and applications} \label{section: Lipschitz filter and applications}

The purpose of this section is to construct a “filter” which converts continuous functions into Lipschitz ones.
As applications, we prove certain “Lipschitz extension theorems” which will be used in the proofs
of Propositions \ref{proposition: C_k(A) is open and dense} and \ref{proposition: C_k(B, B^prime) is open and dense}.
For $u\in \mathbb{R}^n$ and $r>0$ we set 
\[  B_r(u) = \{t\in \mathbb{R}^n\mid |t-u|\leq r\}. \]

\subsection{Preliminaries on Lipschitz functions} \label{subsection: preliminaries on Lipschitz functions}

Let $A$ be a closed subset of $\mathbb{R}^n$. Let $I=[0,1]$ be the unit interval.
We define the space $C(A, I)$ as the space of continuous maps from $A$ to $I$.
This is endowed with the compact-open topology.
For $c\geq 0$ we define a subspace $\lip_c(A, I) \subset C(A, I)$ as the space of 
maps $\varphi\colon A\to I$ that satisfies 
\[ |\varphi(t)-\varphi(u)|\leq c |t-u|, \quad (t, u\in A). \]

\begin{lemma}(Follows from the McShane--Whitney extension theorem (\cite{McSHANE1934ExtensionOR,whitney1934AnalyticEO}), see also \cite[Theorem 2.3]{gutev2020lipschitz}) \label{lemma: elementary Lipschitz extension}
For $\varphi\in \lip_c(A, I)$, set 
\[ \tilde{\varphi}(t) =  \min\{1, \inf\{\varphi(u)+ c|t-u|\mid u\in A\}\}, \quad (t\in \mathbb{R}^n). \]
Then  $\tilde{\varphi}\in \lip_c(\mathbb{R}^n, I)$ and $\tilde{\varphi}(t)  = \varphi(t)$ for $t\in A$.
Moreover the map 
\begin{equation} \label{eq: Lipschitz extension}
   \lip_c(A, I)\ni \varphi\mapsto \tilde{\varphi}\in \lip_c(\mathbb{R}^n, I) 
\end{equation}   
is continuous.
\end{lemma}

\begin{proof}
This is a well-known fact \cite[Chapter 7, p. 100]{Mattila}.
We briefly explain the proof for convenience of readers.
Set 
\[ \psi(t) = \inf\{\varphi(u)+ c|t-u|\mid u\in A\}, \quad (t\in \mathbb{R}^n). \]
Let $t, t^\prime\in \mathbb{R}^n$.
For any $u\in A$ we have 
\[  \psi(t^\prime) \leq \varphi(u) + c|t^\prime-u| \leq \varphi(u) + c|t-u| + c|t^\prime-t|. \]
Taking the infimum over $u\in A$ we get 
\[ \psi(t^\prime) \leq \psi(t) + c|t^\prime-t|. \]
Similarly we have $\psi(t) \leq \psi(t^\prime) + c|t^\prime-t|$.
Hence $|\psi(t)-\psi(t^\prime)|\leq c |t^\prime-t|$.
If $t\in A$ then for any $u\in A$
\[ \varphi(t) -\varphi(u) \leq c |t-u|. \]
Then $\varphi(t) \leq \inf\{\varphi(u) + c|t-u|\mid u\in A\} = \psi(t)$. Therefore $\varphi(t) = \psi(t)$ for $t\in A$.

Since $\tilde{\varphi}(t) = \min\{1, \psi(t)\}$, we also have $\tilde{\varphi}\in \lip_c(\mathbb{R}^n, I)$ and 
$\tilde{\varphi}(t) = \varphi(t)$ for all $t\in A$.

The map (\ref{eq: Lipschitz extension}) is obviously continuous.
\end{proof}

We also need the following elementary lemma.

\begin{lemma} \label{lemma: convex combination}
Let $c>0$, $x\in \mathbb{R}$ and $-c<y<c$.
(We consider that $c$ is fixed and $x,y$ are variables.)
Set 
\[  S = \{s\in [0,1]\mid |(1-s)x+sy|\leq c\}, \quad \mathfrak{s}(x,y) = \inf S. \]
Then $S$ has the form $S = [\mathfrak{s}(x, y), 1]$ and $\mathfrak{s}(x,y)$ is continuous in $x$ and $y$.
\end{lemma}

\begin{proof}
Since $-c<y<c$, we have $1\in S$.
If $x\leq -c$ then $S= \left[-\frac{c+x}{y-x}, 1\right]$.
If $-c\leq x \leq c$ then $S = [0,1]$.
If $x\geq c$ then $S = \left[\frac{x-c}{x-y}, 1\right]$.
\end{proof}

\begin{corollary} \label{cor: convex combination, multiple} 
Let $c_k>0$, $x_k\in \mathbb{R}$ and $-c_k<y_k<c_k$ $(1\leq k \leq m)$.
(We consider that $c_k$ are fixed and $x_k$ and $y_k$ are variables.)
Set 
\[ \mathfrak{s}(x_1, \dots, x_m, y_1, \dots, y_m) = \inf \bigcap_{k=1}^m \{s\in [0,1]\mid |(1-s)x_k+s y_k|\leq c_k\}. \]
Then $\mathfrak{s}(x_1, \dots, x_m, y_1, \dots, y_m)$ is continuous in $x_1, \dots, x_m, y_1, \dots, y_m$.
\end{corollary}

\begin{proof}
Set $\mathfrak{s}_k(x_k, y_k) := \inf\{s\in [0,1]\mid |(1-s)x_k+s y_k|\leq c_k\}$.
We have 
\[ \mathfrak{s}(x_1, \dots, x_m, y_1, \dots, y_m) = \max\{\mathfrak{s}_1(x_1, y_1), \dots, \mathfrak{s}_m(x_m, y_m)\}. \]
Since each $\mathfrak{s}_k(x_k, y_k)$ is continuous, the function $\mathfrak{s}(x_1, \dots, x_m, y_1, \dots, y_m)$ is 
also continuous.
\end{proof}

\subsection{Construction of a filter} \label{subsection: construction of a filter}

Let $M$ be a positive number and set $\mathbb{T}_M = \mathbb{R}^n/M\mathbb{Z}^n$ (the torus).
Let $\pi_M\colon \mathbb{R}^n\to \mathbb{T}_M$ be the natural projection.

We first consider a filter for “discrete signals”:
\begin{lemma} \label{lemma: discrete filter}
Let $c$ and $c^\prime$ be positive numbers with $\frac{1}{M}<c<c^\prime$.
Let $\Lambda\subset \mathbb{T}_M$ be a finite subset and set $\Gamma = \pi_M^{-1}(\Lambda) \subset \mathbb{R}^n$.
There exist a positive number $L$ and 
a $M\mathbb{Z}^n$-equivariant continuous map $G\colon C(\Gamma, I) \to \lip_{c^\prime}(\Gamma, I)$ (where 
$M\mathbb{Z}^n = \{M t\mid t\in \mathbb{Z}^n\}$ naturally acts on $C(\Gamma, I) $ and $\lip_{c^\prime}(\Gamma, I)$ 
by the shift) such that the following two conditions hold.
  \begin{enumerate}
   \item  We have $G(\varphi) = \varphi$ for all $\varphi\in \lip_c(\Gamma, I)$.
   \item  Let $\varphi \in C(\Gamma, I)$ and $\xi \in \Gamma$. If $\varphi(\lambda) = 0$ for all 
   $\lambda \in B_L(\xi)\cap \Gamma$ then $G(\varphi)(\xi) = 0$. 
  \end{enumerate}
\end{lemma}

\begin{proof}
Let $\Lambda = \{\lambda_1, \dots, \lambda_N\}$ and set 
$\Gamma_k = \pi_M^{-1} \{\lambda_1, \lambda_2, \dots, \lambda_k\}$ for $k=1,2,\dots, N$.
Take an ascending chain of positive numbers $c=c_1<c_2<\dots<c_N = c^\prime$.
(Any choice will work.)
We will inductively (on $k=1,2,\dots, N$) define a $M\mathbb{Z}^n$-equivariant continuous map 
$G_k \colon C(\Gamma, I) \to \lip_{c_k}(\Gamma_k, I)$ that satisfies the following two conditions.
  \begin{enumerate}
     \item[(i)] If $\varphi\in \lip_{c}(\Gamma, I)$ then $G_k(\varphi)  = \varphi|_{\Gamma_k}$ (the restriction of $\varphi$ to $\Gamma_k$).
     \item[(ii)] There exists $L_k\geq 0$ such that if $\varphi\in C(\Gamma, I)$ vanishes on $B_{L_k}(\xi)\cap \Gamma_k$ for some
     $\xi \in \Gamma_k$ then $G_k(\varphi)(\xi)  = 0$.  
  \end{enumerate}

Let $\varphi\in C(\Gamma, I)$.
We define $G_1(\varphi) = \varphi|_{\Gamma_1}$ (the restriction of $\varphi$ to $\Gamma_1 = \pi^{-1}(\lambda_1)$).
Since $|t-u|\geq M > \frac{1}{c}$ for distinct $t,u \in \Gamma_1$, we have $G_1(\varphi)\in \lip_{c_1}(\Gamma_1, I)$ and it satisfies the 
induction hypotheses with $L_1 = 0$.

Suppose we have constructed $G_k(\varphi) \in \lip_{c_k}(\Gamma_k, I)$ for some $k<N$. 
We are going to construct $G_{k+1}(\varphi) \in \lip_{c_{k+1}}(\Gamma_{k+1}, I)$.
We define $\psi\in \lip_{c_k}(\Gamma_{k+1}, I)$ (by using Lemma \ref{lemma: elementary Lipschitz extension}) by 
\[ \psi(t) = \min\left\{1, \inf\{ G_k(\varphi)(u) + c_k|t-u|\mid u\in \Gamma_k\}\right\} \quad 
    (t\in \Gamma_{k+1}). \]
We have $\psi|_{\Gamma_k} = G_k(\varphi)$.
For each $t\in \Gamma_{k+1}\setminus \Gamma_k = \pi^{-1}(\lambda_{k+1})$, we define
$\mathfrak{s}(t) \in [0, 1]$ as the minimum number of $s\in [0, 1]$ satisfying the following condition:
\begin{equation} \label{eq: choice of s(t)}
  \forall u\in \Gamma_k: \quad \left|(1-s)\varphi(t)+ s\psi(t)-G_k(\varphi)(u)\right|\leq c_{k+1} |t-u|. 
\end{equation}
This condition seems to involve infinitely many inequalities.
However the relevant $u\in \Gamma_k$ are only those satisfying $|t-u|\leq \frac{1}{c_{k+1}}$.
(If $|t-u| > \frac{1}{c_{k+1}}$ then the right-hand side is greater than one.)
Therefore (\ref{eq: choice of s(t)}) essentially contains only finitely many inequalities.

The condition (\ref{eq: choice of s(t)}) is equivalent to
\[ \forall u\in \Gamma_k: \quad 
   \left|(1-s)\left(\varphi(t)- G_k(\varphi)(u)\right)+ s\left(\psi(t)-G_k(\varphi)(u)\right)\right|\leq c_{k+1} |t-u|.  \]
Moreover we have 
\[ \left|\psi(t)-G_k(\varphi)(u)\right| = \left|\psi(t) - \psi(u)\right| \leq c_k |t-u| < c_{k+1} |t-u|. \]
Then $\mathfrak{s}(t)$ is continuous in $\varphi \in C(\Gamma, I)$ by Corollary \ref{cor: convex combination, multiple}.
Now we define $G_{k+1}(\varphi) \colon \Gamma_{k+1} \to I$ by 
\[ G_{k+1}(\varphi)(t) = \begin{cases}  
                                 G_k(\varphi)(t)  & (t\in \Gamma_k) \\
                                 (1-\mathfrak{s}(t))\varphi(t) + \mathfrak{s}(t) \psi(t) & (t\in \Gamma_{k+1}\setminus \Gamma_k)
                                 \end{cases}. \] 

For $t\in \Gamma_{k+1}\setminus \Gamma_k$ and $u\in \Gamma_k$ we have 
$|G_{k+1}(\varphi)(t) - G_{k+1}(\varphi)(u)| = |G_{k+1}(\varphi)(t) - G_k(\varphi)(u)| \leq c_{k+1}|t-u|$ by (\ref{eq: choice of s(t)}).
For $t, u \in \Gamma_k$ we have 
$|G_{k+1}(\varphi)(t) - G_{k+1}(\varphi)(u)| = |G_{k}(\varphi)(t) - G_{k}(\varphi)(u)| \leq c_k |t-u| \leq c_{k+1}|t-u|$.
For distinct $t, u \in \Gamma_{k+1}\setminus \Gamma_k$ we have $|t-u| \geq M > \frac{1}{c_1}$ and hence 
$|G_{k+1}(\varphi)(t)-G_{k+1}(\varphi)(u)| \leq 1 \leq c_{k+1}|t-u|$.
Therefore we have $G_{k+1}(\varphi) \in \lip_{c_{k+1}}(\Gamma_{k+1}, I)$.

We are going to check that $G_{k+1}(\varphi)$ satisfies the conditions (i) and (ii).
If $\varphi\in \lip_c(\Gamma, I)$ then $G_k(\varphi) = \varphi|_{\Gamma_k}$ and 
$\mathfrak{s}(t) = 0$ for all $t\in \Gamma_{k+1}\setminus \Gamma_k$.
Hence $G_{k+1}(\varphi) = \varphi|_{\Gamma_{k+1}}$.
This has shown (i).
Next we consider (ii).
Set $L_{k+1} = L_k + \frac{1}{c_{k+1}}$.
Suppose that $\varphi\in C(\Gamma, I)$ vanishes on $B_{L_{k+1}}(\xi)\cap \Gamma_{k+1}$ for some $\xi \in \Gamma_{k+1}$.
Then for any $\eta \in B_{1/{c_{k+1}}}(\xi)\cap \Gamma_k$ we have 
$\varphi(t) = 0$ for $t\in B_{L_k}(\eta)\cap \Gamma_k$.
By the induction hypothesis, we have $G_k(\varphi)(\eta) = 0$ for $\eta \in B_{1/{c_{k+1}}}(\xi)\cap \Gamma_k$.
Together with $\varphi(\xi) = 0$, this implies $\mathfrak{s}(\xi)  = 0$ and hence $G_{k+1}(\varphi)(\xi)=0$

Now we have completed the induction process. We finally set $G(\varphi) := G_N(\varphi)$ and $L:=L_N$.
\end{proof}

The next proposition provides us a convenient filter that converts continuous functions into Lipschitz ones.
This is the main result of this subsection.

\begin{proposition} \label{proposition: Lipschitz filter}
Let $\varepsilon>0$ and $0<c<c^\prime$.
There exist $R>0$ and an $\mathbb{R}^n$-equivariant continuous map 
$F\colon C(\mathbb{R}^n, I)\to \lip_{c^\prime}(\mathbb{R}^n, I)$ satisfying the following three conditions.
  \begin{enumerate}
   \item For $\varphi\in \lip_c(\mathbb{R}^n, I)$ we have $\left|F(\varphi)(t) - \varphi(t)\right| < \varepsilon$ for all $t\in \mathbb{R}^n$.
   \item If $\varphi \in C(\mathbb{R}^n, I)$ is a constant function, then $F(\varphi) = \varphi$.
   \item If $\varphi \in C(\mathbb{R}^n, I)$ vanishes on $B_R(\xi)$ for some $\xi \in \mathbb{R}^n$, then $F(\varphi)(\xi)  = 0$.  
  \end{enumerate}
\end{proposition}

\begin{proof}
Take a positive number $M$ with $\frac{1}{M} < c$.
We choose a $(\frac{\varepsilon}{2c^\prime})$-dense finite subset $\Lambda\subset \mathbb{T}_M$.
Here “$(\frac{\varepsilon}{2c^\prime})$-dense” means that 
(setting $\Gamma := \pi_M^{-1}(\Lambda)$) for any $t\in \mathbb{R}^n$ there 
exists $u\in \Gamma$ satisfying $|t-u|< \frac{\varepsilon}{2c^\prime}$.

By Lemma \ref{lemma: discrete filter}, we can construct a positive number $L$ and an
$M\mathbb{Z}^n$-equivariant continuous map 
$G\colon C(\Gamma, I)\to \lip_{c^\prime}(\Gamma, I)$ such that 
  \begin{enumerate}
    \item[(i)] for $\varphi \in \lip_c(\Gamma, I)$ we have $G(\varphi) = \varphi$,
    \item[(ii)] if $\varphi\in C(\Gamma, I)$ vanishes on $\Gamma \cap B_L(\xi)$ for some $\xi \in \Gamma$ then 
                  we have $G(\varphi)(\xi)  = 0$.
  \end{enumerate}

Let $\varphi\in C(\mathbb{R}^n, I)$.
We denote by $\varphi|_{\Gamma}$ the restriction of $\varphi$ to $\Gamma$ and consider 
$G(\varphi|_\Gamma) \in \lip_{c^\prime}(\Gamma, I)$.
Set $J = \left[0, 1+\frac{\varepsilon}{2}\right]$.
We define $F_0(\varphi)\in \lip_{c^\prime}(\mathbb{R}^n, J)$ by 
\[ F_0(\varphi)(t) = \inf\{G(\varphi|_\Gamma)(u) + c^\prime|t-u|\mid u\in \Gamma\} \quad (t\in \mathbb{R}^n). \]
We have $F_0(\varphi)(t) \leq 1+ \frac{\varepsilon}{2}$ 
because $\Gamma$ is $\left(\frac{\varepsilon}{2c^\prime}\right)$-dense in $\mathbb{R}^n$.
$F_0\colon C(\mathbb{R}^n, I)\to \lip_{c^\prime}(\mathbb{R}^n, J)$ is an $M\mathbb{Z}^n$-equivariant continuous map.
Define $\psi\colon \mathbb{R}^n\to \mathbb{R}$ by 
$\psi(t) = c^\prime \inf_{u\in \Gamma} |t-u|$.
We have 
\begin{equation}  \label{eq: psi and F_0 varphi}
   \psi(t) \leq F_0(\varphi)(t) \leq 1+\psi(t)
\end{equation}   
for all $\varphi\in C(\mathbb{R}^n, I)$ and
$t\in  \mathbb{R}^n$.
If $\varphi \in \lip_c(\mathbb{R}^n, I)$ then we have $F_0(\varphi)(u) = \varphi(u)$ for all $u\in \Gamma$.
Hence for any $t\in \mathbb{R}^n$, taking $u\in \Gamma$ with $|t-u| < \frac{\varepsilon}{2c^\prime}$, we have 
\begin{equation} \label{eq: the difference between F_0 varphi and varphi}
   |F_0(\varphi)(t) - \varphi(t)|\leq |F_0(\varphi)(t)-F_0(\varphi)(u)| + |\varphi(u)-\varphi(t)| < \varepsilon. 
\end{equation}   
If $\varphi\in C(\mathbb{R}^n, I)$ is a constant function then $F_0(\varphi) = \varphi + \psi$.

Set $A = \frac{1}{M^n}\int_{[0, M)^n}\psi(u)\, du$ where $du$ denotes the standard Lebesgue measure on $\mathbb{R}^n$.

\begin{claim} \label{claim: invariance of integral}
For any $t\in \mathbb{R}^n$ we have 
\[  \frac{1}{M^n} \int_{[0, M)^n} \psi(t-u)\, du = A. \]
\end{claim}
\begin{proof}
$\psi$ is an $M\mathbb{Z}^n$-invariant function.
Let $s = t - (M, \dots, M)$. We have 
\[ \frac{1}{M^n} \int_{[0, M)^n} \psi(t-u)\, du = \frac{1}{M^n} \int_{[0, M)^n} \psi(s+u) \, du 
    = \frac{1}{M^n} \int_{s+[0, M)^n} \psi(u) \, du. \]
Since $[0, M)^n$ is a fundamental domain for the lattice $M\mathbb{Z}^n$, we can find a measurable partition
$s+[0, M)^n = K_1\cup \dots \cup K_{2^n}$ and $a_i\in M\mathbb{Z}^n$ $(1\leq i \leq 2^n)$ such that 
$[0, M)^n = \bigcup_{i=1}^{2^n}(a_i + K_i)$ (disjoint union).
Then, by using the $M\mathbb{Z}^n$-invariance of $\psi$,
\[  \int_{s+[0, M)^n} \psi(u) \, du = \sum_{i=1}^{2^n} \int_{K_i} \psi(u) \, du = \sum_{i=1}^{2^n} \int_{a_i+K_i} \psi(u)\, du 
     = \int_{[0, M)^n} \psi(u)\, du. \]
\end{proof}

Let $\varphi\in C(\mathbb{R}^n, I)$. We define
\[ F(\varphi)(t) = \frac{1}{M^n} \int_{[0, M)^n} F_0(\sigma^u \varphi)(t-u)\, du - A, \quad (t\in \mathbb{R}^n), \]
where $\sigma^u \varphi$ is the shift of $\varphi$ (i.e. $(\sigma^u\varphi)(t) = \varphi(t+u)$).
Since $\psi(t-u) \leq F_0(\sigma^u \varphi)(t-u) \leq 1 + \psi(t-u)$ by (\ref{eq: psi and F_0 varphi}), 
we have $0\leq F(\varphi)(t) \leq 1$.
Therefore $F(\varphi)\in \lip_{c^\prime}(\mathbb{R}^n, I)$.
If $\varphi \in \lip_c(\mathbb{R}^n, I)$ then 
\begin{equation*}
   \begin{split}
   \left|F(\varphi)(t) - \varphi(t)\right| & = \frac{1}{M^n}\left|\int_{[0,M)^n} \left(F_0(\sigma^u \varphi)(t-u) - (\sigma^u\varphi)(t-u)\right)du \right|\\
   & \leq \frac{1}{M^n}\int_{[0,M)^n}\left|F_0(\sigma^u \varphi)(t-u) - (\sigma^u\varphi)(t-u)\right| du \\
   & < \varepsilon  \quad  (\text{by (\ref{eq: the difference between F_0 varphi and varphi})}).
   \end{split}
\end{equation*}   
If $\varphi$ is a constant function then 
\[ F(\varphi)(t) = \varphi(t) +  \frac{1}{M^n} \int_{[0, M)^n} \psi(t-u)\, du - A = \varphi(t). \]
We have shown that $F$ satisfies the required conditions (1) and (2).

\begin{claim}
$F\colon C(\mathbb{R}^n, I) \to \lip_{c^\prime}(\mathbb{R}^n, I)$ is an $\mathbb{R}^n$-equivariant continuous map.
\end{claim}

\begin{proof}
The continuity is an immediate consequence of the definition.
We are going to prove that it is $\mathbb{R}^n$-equivariant. This is similar to the proof of Claim \ref{claim: invariance of integral}.
Let $s, t\in \mathbb{R}^n$.
\begin{align*}
   F(\sigma^s \varphi)(t) & = \frac{1}{M^n}\int_{[0, M)^n} F_0\left(\sigma^{s+u}\varphi\right)(t-u)\, du -A \\
    & =  \frac{1}{M^n}\int_{s+[0, M)^n} F_0\left(\sigma^{u}\varphi\right)(t-u+s)\, du -A.
\end{align*}    
There exist a measurable partition $[0, M)^n = K_1\cup \dots \cup K_{2^n}$ and 
$a_1, \dots, a_{2^n}\in M\mathbb{Z}^n$ such that 
$s+[0,M)^n = \bigcup_{i=1}^{2^n} (a_i + K_i)$ (disjoint union).
Then 
\begin{align*}
  \int_{s+[0, M)^n} F_0\left(\sigma^{u}\varphi\right)(t-u+s)\, du & 
  = \sum_{i=1}^{2^n} \int_{a_i+K_i}  F_0\left(\sigma^u \varphi\right)(t-u+s)\, du \\
  & = \sum_{i=1}^{2^n} \int_{K_i} F_0\left(\sigma^{a_i+u}\varphi\right)(t-a_i - u + s)\, du.
\end{align*}
Since $F_0$ is $M\mathbb{Z}^n$-equivariant, we have
$F_0\left(\sigma^{a_i+u}\varphi\right)(t-a_i - u + s) = F_0\left(\sigma^u \varphi\right)(t-a_i-u+s+a_i) = 
F_0(\sigma^u \varphi)(t-u+s)$.
Therefore 
\begin{align*}
    \int_{s+[0, M)^n} F_0\left(\sigma^{u}\varphi\right)(t-u+s)\, du  & = 
    \sum_{i=1}^{2^n} \int_{K_i} F_0\left(\sigma^{u}\varphi\right)(t - u + s)\, du \\
    & = \int_{[0, M)^n} F_0(\sigma^u \varphi)(t-u+s)\, du. 
\end{align*}    
Thus 
\[  F(\sigma^s \varphi)(t)  = \frac{1}{M^n} \int_{[0, M)^n} F_0(\sigma^u \varphi)(t-u+s)\, du  - A  = F(\varphi)(t+s). \]
\end{proof}
Finally we are going to check the condition (3).
Let $R = L + \frac{\varepsilon}{2c^\prime} + 2\sqrt{n}M$.
Suppose that $\varphi\in C(\mathbb{R}^n, I)$ vanishes on $B_R(\xi)$ for some $\xi \in \mathbb{R}^n$.
Let $u\in [0, M)^n$. 
We have $(\sigma^u \varphi)(t) = 0$ for $t\in B_{L+\frac{\varepsilon}{2c^\prime}+ \sqrt{n}M}(\xi)$.
Let $\eta\in B_{\sqrt{n}M}(\xi)$. 
Since $\Gamma$ is $\left(\frac{\varepsilon}{2c^\prime}\right)$-dense in $\mathbb{R}^n$,
there exists $t\in \Gamma \cap B_{\frac{\varepsilon}{2c^\prime}}(\eta)$ with 
$\psi(\eta) = c^\prime |t-\eta|$. The function $\sigma^u \varphi$ vanishes on $B_L(t)\cap \Gamma$, and hence 
$G\left((\sigma^u \varphi)|_\Gamma\right)(t)=0$ by the condition (ii) of the map $G$. 
Therefore $F_0(\sigma^u\varphi)(\eta) = \psi(\eta)$.
It follows that 
\[ F(\varphi)(\xi)  = \frac{1}{M^n}\int_{[0, M)^n} F_0(\sigma^u \varphi)(\xi-u)\, du - A = \frac{1}{M^n}\int_{[0, M)^n}\psi(\xi-u)\, du - A =0. \]
\end{proof}

\subsection{Applications} \label{subsection: applications}

Here we present applications of Proposition \ref{proposition: Lipschitz filter}.
Throughout this subsection we assume that $X$ is a compact metrizable space and $T\colon \mathbb{R}^n\times X\to X$
is a continuous action of $\mathbb{R}^n$ on it.

\begin{lemma} \label{lemma: filtering cutoff}
Let $A$ be a $T$-invariant closed subset of $X$, and let $V$ be an open subset of $X$ with $A\subset V$.
For any $\delta>0$ there exists an $\mathbb{R}^n$-equivariant continuous map 
$\beta\colon X\to \lip_\delta(\mathbb{R}^n, I)$ such that 
  \begin{enumerate}
    \item $\beta(x)(t) = 1$ for any $x\in A$ and $t\in \mathbb{R}^n$ (where $\beta(x)(t)$ denotes the value of
    the function $\beta(x)\in \lip_\delta(\mathbb{R}^n, I)$ at $t\in \mathbb{R}^n$), 
    \item  $\beta(x)(0)= 0$ for any $x\in X\setminus V$.  
  \end{enumerate}
\end{lemma}

\begin{proof}
By Proposition \ref{proposition: Lipschitz filter}, we can find $R>0$ and an equivariant continuous map 
$F\colon C(\mathbb{R}^n, I)\to \lip_\delta(\mathbb{R}^n, I)$ that satisfies the following two conditions.
   \begin{enumerate}
     \item[(i)] If $\varphi\in C(\mathbb{R}^n, I)$ is a constant function then $F(\varphi) = \varphi$.
     \item[(ii)] If $\varphi\in C(\mathbb{R}^n, I)$ vanishes on $B_R(\xi)$ for some $\xi \in \mathbb{R}^n$ then 
                    $F(\varphi)(\xi) = 0$.    
   \end{enumerate}
Since $A$ is a $T$-invariant closed subset, 
we can find an open subset $W$ with $A\subset W\subset V$ such that for any $x\in W$ and $t\in B_R(0)$ we have 
$T^t x\in V$.
Let $\alpha\colon X\to [0,1]$ be a continuous function such that $\alpha= 1$ on $A$ and $\supp (\alpha) \subset W$. 
We define an equivariant continuous map $\tilde{\alpha}\colon X\to C(\mathbb{R}^n, I)$ by 
$\tilde{\alpha}(x)(t) = \alpha(T^t x)$.

We define an equivariant continuous map $\beta\colon X\to \lip_\delta(\mathbb{R}^n, I)$ by 
$\beta(x) = F\left(\tilde{\alpha}(x)\right)$.
If $x\in A$ then $\tilde{\alpha}(x)$ is constantly equal to $1$. Hence $\beta(x)$ is also constantly equal to $1$.
If $x\in X\setminus V$ then for any $t\in B_R(0)$ we have $T^t x \not\in W$ and hence $\alpha(T^t x) = 0$.
Therefore $\tilde{\alpha}(x)(t) = 0$ for all $t\in B_R(0)$.
It follows from the property (ii) of the filter $F$ that we have $\beta(x)(0) = F\left(\tilde{\alpha}(x)\right)(0) = 0$. 
\end{proof}

\begin{lemma} \label{lemma: Lipschitz approximation}
Let $0<c<c^\prime$ and $\varepsilon>0$. Let $A$ be a $T$-invariant closed subset of $X$.
Suppose that $f$ and $g$ are equivariant continuous maps from $X$ to $\lip_c(\mathbb{R}^n, I)$ such that 
$\left|f(x)(t)-g(x)(t)\right| < \varepsilon$ for all $x\in A$ and $t\in \mathbb{R}^n$.
Then there exists an equivariant continuous map $h\colon X\to \lip_{c^\prime}(\mathbb{R}^n, I)$ such that 
$h(x) = g(x)$ for all $x\in A$ and that $\left|h(x)(t)-f(x)(t)\right|< \varepsilon$ for all $x\in X$ and $t\in \mathbb{R}^n$.
\end{lemma}

\begin{proof}
We take an open set $V$ of $X$ such that $A\subset V$ and $\left|f(x)(0)- g(x)(0)\right|<\varepsilon$ for all 
$x\in V$.
We take $\delta>0$ with $c+2\delta<c^\prime$.

By Lemma \ref{lemma: filtering cutoff}, we can find an equivariant continuous map 
$\beta\colon X\to \lip_\delta(\mathbb{R}^n, I)$ that satisfies the following two conditions.
   \begin{enumerate}
     \item  $\beta(x)$ is constantly equal to $1$ for all $x\in A$.
     \item  $\beta(x)(0) = 0$ for all $x \in X\setminus V$.
   \end{enumerate}
For $x\in X$ we define $h(x)\colon \mathbb{R}^n\to I$ by 
\[ h(x)(t) = \left(1-\beta(x)(t)\right)\cdot f(x)(t) + \beta(x)(t)\cdot g(x)(t). \]
For $x\in A$ we have $h(x)(t) = g(x)(t)$.
Let $x\in X\setminus A$. 
We have $h(x)(t) = \left(1-\beta(T^t x)(0)\right)\cdot f(T^t x)(0) + \beta(T^t x)(0) \cdot g(T^t x)(0)$.
If $T^t x\in X\setminus V$ then $\beta(T^t x) (0) = 0$ and hence $h(x)(t) = f(T^t x)(0) = f(x)(t)$.
If $T^t x\in V$ then $|f(T^t x)(0) - g(T^t x)(0)| < \varepsilon$ and hence 
\[ \left|h(x)(t) - f(x)(t)\right| = \left|\beta(T^t x)(0)\left(g(T^t x)(0) - f(T^t x)(0)\right)\right| < \varepsilon. \]

We are going to check that $h(x)\in \lip_{c^\prime}(\mathbb{R}^n, I)$ for all $x\in X$.
\begin{align*}
  h(x)(t) - h(x)(u) = & \left(1-\beta(x)(t)\right)\cdot f(x)(t) + \beta(x)(t)\cdot g(x)(t) 
                          - \left(1-\beta(x)(u)\right)\cdot f(x)(u) \\
                          & - \beta(x)(u)\cdot g(x)(u) \\
   = & \left(1-\beta(x)(t)\right)\cdot \left(f(x)(t) - f(x)(u)\right) \\
      &+ \left(\beta(x)(u)- \beta(x)(t)\right)\cdot \left(f(x)(u)-g(x)(u)\right) \\
       &   + \beta(x)(t) \left(g(x)(t) - g(x)(u)\right). 
\end{align*}
Since $f(x)$ and $g(x)$ are $c$-Lipschitz and $\beta(x)$ is $\delta$-Lipschitz,
\begin{align*}
 \left|h(x)(t)-h(x)(u)\right| \leq & \left(1-\beta(x)(t)\right)\cdot \left|f(x)(t) - f(x)(u)\right| \\
      &+ \left|\beta(x)(u)- \beta(x)(t)\right|\cdot \left(\left|f(x)(u)\right|+\left|g(x)(u)\right|\right) \\
       &   + \beta(x)(t) \left|g(x)(t) - g(x)(u)\right|  \\
       \leq & \left(1-\beta(x)(t)\right)\cdot c|t-u| +2\delta|t-u| + \beta(x)(t)\cdot c|t-u|  \\
       \leq & (c+2\delta) |t-u| \\
        \leq & c^\prime |t-u|.
\end{align*}
\end{proof}

Recall that $\fix(X, T) = \{x\in X\mid T^t x = x \, (\forall t\in \mathbb{R}^n)\}$.

\begin{lemma}  \label{lemma: Lipschitz extension}
Let $0<c<c^\prime$ and $\varepsilon>0$.
Let $A$ be a closed invariant subset of $X$.
For any equivariant continuous map $g\colon A\to \lip_c(\mathbb{R}^n, I)$ there exists an 
equivariant continuous map $g^\prime\colon X\to \lip_{c^\prime}(\mathbb{R}^n, I)$ such that 
  \begin{enumerate}
    \item $\left|g(x)(t)-g^\prime(x)(t)\right| < \varepsilon$ for all $x\in A$ and $t\in \mathbb{R}^n$,
    \item $g^\prime(x) = g(x)$ for all $x\in A\cap \fix(X, T)$.
  \end{enumerate}
\end{lemma}

\begin{proof}
By Proposition \ref{proposition: Lipschitz filter}, we can construct an equivariant continuous map 
$F\colon C(\mathbb{R}^n, I) \to \lip_{c^\prime}(\mathbb{R}^n, I)$ that satisfies the following two conditions.
 \begin{enumerate}
   \item[(i)] For $\varphi\in \lip_c(\mathbb{R}^n, I)$ we have $\left|F(\varphi)(t)-\varphi(t)\right| < \varepsilon$ for all $t\in \mathbb{R}^n$.
   \item[(ii)] If $\varphi\in C(\mathbb{R}^n, I)$ is a constant function, then $F(\varphi) = \varphi$.
 \end{enumerate}
We define $g_0\colon A\to I$ by $g_0(x) = g(x)(0)$.
By the Tietze extension theorem, we can extend $g_0$ to a continuous function $g_0\colon X\to I$.
Then we extend $g$ to an equivariant continuous map $g\colon X\to C(\mathbb{R}^n, I)$ by 
$g(x)(t) = g_0(T^t x)(0)$.

We define $g^\prime \colon X\to \lip_{c^\prime}(\mathbb{R}^n, I)$ by 
$g^\prime(x) = F\left(g(x)\right)$.
If $x\in A$ then $g(x)\in \lip_c(\mathbb{R}^n, I)$ and hence 
$\left|g^\prime(x)(t)-g(x)(t)\right| = \left|F\left(g(x)\right)(t) - g(x)(t)\right| < \varepsilon$ for all $t\in \mathbb{R}^n$.
Moreover, if $x\in A\cap \fix(X, T)$, then $g(x)$ is a constant function and hence $g^\prime(x) = F\left(g(x)\right) = g(x)$.
\end{proof}

The next proposition is the main result of this subsection.
This will be used in the proofs of Propositions \ref{proposition: C_k(A) is open and dense}
and \ref{proposition: C_k(B, B^prime) is open and dense}.

\begin{proposition}  \label{prop: extension and filter}
Let $0<c<c^\prime$, and $\varepsilon_1>0$ and $\varepsilon_2>0$.
Let $A$ be a $T$-invariant closed subset of $X$.
Let $f\colon X\to \lip_c(\mathbb{R}^n, I)$ and $g\colon A\to \lip_c(\mathbb{R}^n, I)$ be equivariant continuous maps that satisfy 
$\left|f(x)(t)- g(x)(t)\right|<\varepsilon_1$ for all $x\in A$ and $t\in \mathbb{R}^n$.
Then there exists an equivariant continuous map $h\colon X\to \lip_{c^\prime}(\mathbb{R}^n, I)$ such that
  \begin{enumerate}
    \item $\left|h(x)(t)-f(x)(t)\right|<\varepsilon_1$ for all $x\in X$ and $t\in \mathbb{R}^n$,
    \item $\left|h(x)(t)- g(x)(t)\right| < \varepsilon_2$ for all $x\in A$ and $t\in \mathbb{R}^n$,
    \item $h(x) = g(x)$ for all $x\in A\cap \fix(X, T)$.
  \end{enumerate}
\end{proposition}

\begin{proof}
Take $\varepsilon \in (0, \varepsilon_2)$ such that $|f(x)(0) - g(x)(0)|+\varepsilon < \varepsilon_1$ for all $x\in A$.
Since $A$ is $T$-invariant, we also have $\left|f(x)(t)-g(x)(t)\right|+ \varepsilon < \varepsilon_1$ for all $x\in A$ and 
$t\in \mathbb{R}^n$.
Take arbitrary $c_1\in (c, c^\prime)$, say $c_1 = \frac{c+c^\prime}{2}$.

By applying Lemma \ref{lemma: Lipschitz extension} to $g$, we can construct an equivariant continuous map 
$g^\prime\colon X\to \lip_{c_1}(\mathbb{R}^n, I)$ such that 
  \begin{enumerate}
    \item[(i)] $\left|g^\prime(x)(t) - g(x)(t)\right| < \varepsilon$ for all $x\in A$ and $t\in \mathbb{R}^n$,
    \item[(ii)] $g^\prime(x)= g(x)$ for all $x\in A \cap \fix(X, T)$.  
  \end{enumerate}
It follows from (i) that for $x\in A$ and $t\in \mathbb{R}^n$
\[  \left|g^\prime(x)(t)-f(x)(t)\right| < \varepsilon + \left|g(x)(t)-f(x)(t)\right| < \varepsilon_1. \]   
We apply Lemma \ref{lemma: Lipschitz approximation} to $f$ and $g^\prime$. 
Then we get 
an equivariant continuous map $h\colon X\to \lip_{c^\prime}(\mathbb{R}^n, I)$ such that 
  \begin{enumerate}
     \item[(a)]  $h(x) = g^\prime(x)$ for all $x\in A$,
     \item[(b)]  $\left|h(x)(t) - f(x)(t)\right| < \varepsilon_1$ for all $x\in X$ and $t\in \mathbb{R}^n$.
  \end{enumerate}
It follows from the above (i) and (a) that for all $x\in A$ and $t\in \mathbb{R}^n$
\[ \left|h(x)(t) - g(x)(t)\right| = \left|g^\prime(x)(t) - g(x)(t)\right| < \varepsilon  < \varepsilon_2. \]
For $x\in A \cap \fix(X, T)$, we have $h(x) = g^\prime(x) = g(x)$.
\end{proof}

\section{Construction of local perturbation maps} \label{section: construction of local perturbation maps}

Let $\Omega\subset \mathbb{R}^n$ and $\varphi\colon \Omega\to \mathbb{R}$.
For $p\in \Omega$ and $r>0$ with $B_r(p) \subset \Omega$, we define
\[  L(\varphi, p, r)  = \sup_{0<|t-p|\leq r} \frac{|\varphi(t)-\varphi(p)|}{|t-p|}. \] 

For $r>0$ we denote $B_r(0) = \{t\in \mathbb{R}^n\mid |t|\leq r\}$ by $B_r$.
We sometimes denote it by $B_r(\mathbb{R}^n)$ for clarifying the ambient space $\mathbb{R}^n$.
(We will consider $B_r(\mathbb{R}^k)$ for $k\leq n$ in \S \ref{section: local section} and \S \ref{section: proof of main propositions}.) 
We also denote $\partial B_r = \{t\in \mathbb{R}^n\mid |t|=r\}$.

The next proposition will provide us a “local model” for constructing a perturbation map in the proof of 
Proposition \ref{proposition: C_k(A) is open and dense}.

\begin{proposition} \label{prop: local perturbation map, no translation invariance}
Let $r>0$, $\varepsilon>0$ and $0<c<c^\prime$.
There exists a continuous map $G\colon \lip_c(B_r, I) \to \lip_{c^\prime}(B_r, I)$ that satisfies the 
following conditions.
  \begin{enumerate}
    \item $\left|G(\varphi)(t) - \varphi(t)\right|< \varepsilon$ for all $\varphi\in \lip_c(B_r, I)$ and $t\in B_r$.
    \item $G(\varphi)(t) = \varphi(t)$ for all $\varphi\in \lip_c(B_r, I)$ and $t\in \partial B_r$.
    \item For any $\varphi \in \lip_c(B_r, I)$, $G(\varphi)$ has no “translation invariance”.
    Namely, for any line $L\subset \mathbb{R}^n$ through the origin, there exist distinct two points $t, u \in B_r$ that satisfy 
    $t-u \in L$ and $G(\varphi)(t) \neq G(\varphi)(u)$.
  \end{enumerate}
\end{proposition}

\begin{proof}
We can assume $0<\varepsilon < 1$.
Set $c_1 = \frac{c+c^\prime}{2}$. We take $\varepsilon_1\in (0, \frac{\varepsilon}{3})$ so small that 
there exists $\psi\in \lip_{\frac{c^\prime-c}{2}}(B_r, I)$ satisfying $0\leq \psi \leq \varepsilon_1$ all over $B_r$, 
$\psi(t) = \varepsilon_1$ for $t\in B_{\frac{r}{2}}$ and $\psi(t) = 0$ for $t\in \partial B_r$.

Suppose we are given $\varphi\in \lip_c(B_r, I)$.
Set $\varphi_1(t) = \max\{0, \varphi(t)-\psi(t)\}$.
We have $\varphi_1  \in \lip_{c_1}(B_r, I)$, $|\varphi_1(t)-\varphi(t)|\leq \varepsilon_1$ all over $B_r$,
$\varphi_1(t) = \varphi(t)$ for $t\in \partial B_r$ and $0\leq \varphi_1\leq 1-\varepsilon_1$ on $B_{r/2}$.
We take $0<\delta<\min\{\frac{r}{4}, \frac{\varepsilon_1}{c^\prime}\}$. 

We define $G(\varphi) \in \lip_{c^\prime}(B_r, I)$ as follows.
  \begin{enumerate}
     \item[(i)]   First we set $G(\varphi)(t)  = \varphi_1(t)$ for $\delta\leq |t|\leq r$.
     \item[(ii)]  Next, for $|t| < \delta$ we set
     \[ G(\varphi)(t) = \min\{1, \inf\{\varphi_1(u) + c^\prime|u-t|\mid u \in  \partial B_\delta\}\}. \]
  \end{enumerate}  
We have $G(\varphi)(0) = \min_{t\in \partial B_\delta} \varphi_1(t) + c^\prime \delta$
because $c^\prime \delta < \varepsilon_1$ and $\varphi_1 \leq 1-\varepsilon_1$ on $B_{\frac{r}{2}}$.
Take $t_0 \in \partial B_\delta$ achieving the minimum of $\varphi_1(t)$ over $\partial B_\delta$.
Then $G(\varphi)(0) = \varphi_1(t_0) + c^\prime \delta = G(\varphi)(t_0) + c^\prime \delta$ by (i), and hence 
$L\left(G(\varphi), 0, \delta\right) \geq c^\prime$ (since $|t_0|=\delta$).
On the other hand it follows from Lemma \ref{lemma: elementary Lipschitz extension} that 
$L\left(G(\varphi), 0, \delta\right) \leq c^\prime$. Therefore we have
\begin{equation} \label{eq: local Lipschitz constant at the origin, translation invariance}
  L\left(G(\varphi), 0, \delta\right) = c^\prime.
\end{equation}
     
We claim that $|G(\varphi)(t)-\varphi_1(t)| < 2\varepsilon_1$.
This is obvious for $\delta\leq |t|\leq r$. 
For $|t| < \delta$, we take $u\in \partial B_\delta$ with $|t-u|\leq \delta$
and estimate
\begin{align*}
  |G(\varphi)(t)-\varphi_1(t)| & \leq |G(\varphi)(t)-G(\varphi)(u)| + |\varphi_1(u)-\varphi_1(t)|, 
  \quad (\text{by } G(\varphi)(u)= \varphi_1(u)) \\
  & \leq 2 c^\prime |t-u|  \\
  & \leq 2 c^\prime \delta \\
  & \leq 2\varepsilon_1.
\end{align*}      
Therefore we have 
\[    |G(\varphi)(t)-\varphi(t)| \leq   |G(\varphi)(t)-\varphi_1(t)| + |\varphi_1(t)-\varphi(t)| < 3\varepsilon_1 < \varepsilon. \]
We also have $G(\varphi)(t) = \varphi_1(t) = \varphi(t)$ for $t\in \partial B_r$. 
     
Finally we are going to prove that $G(\varphi)$ has no translation invariance.
Let $L\subset \mathbb{R}^n$ be a line through the origin.
Take $p\in L$ with $|p| = \frac{3r}{4}$. 
Then $B_\delta(p)\subset B_r\setminus B_{\frac{r}{2}}$ by $\delta < \frac{r}{4}$,
and hence $B_\delta(p) \cap B_\delta =\emptyset$.
Therefore $G(\varphi)(t) = \varphi_1(t)$ for $t\in B_\delta(p)$.
Since $\varphi_1\in \lip_{c_1}(B_r, I)$, we have 
\[  L\left(G(\varphi), p, \delta\right) \leq c_1 < c^\prime. \]
If $G(\varphi)$ has the translation invariance in the direction of $L$, then we have
\[ G(\varphi)(t+p)  = G(\varphi)(t) \quad (\forall t\in B_\delta). \]
This implies $L\left(G(\varphi), p, \delta\right) =  L\left(G(\varphi), 0, \delta\right)$.
However we have $ L\left(G(\varphi), 0, \delta\right) = c^\prime$ by (\ref{eq: local Lipschitz constant at the origin, translation invariance}).
This is a contradiction.
\end{proof}

\begin{lemma} \label{lemma: multi bump function}
Let $r>0$, $\varepsilon>0$ and $0<c_1<c_2<c_3<c_4$. 
Let $p_1,\dots, p_N\in B_{\frac{r}{4}}\setminus \{0\}$ be distinct points.
There exist a positive number $\delta$ with 
\begin{equation} \label{eq: choice of delta in multi-bumpy function}
   \delta < \min\left\{\frac{r}{4}, \min_{k} \frac{|p_k|}{3}, \min_{k \neq \ell} \frac{|p_k-p_\ell|}{3}\right\} 
\end{equation}   
and a continuous map 
\[ G\colon \lip_{c_1}(B_r, I)\times I^N \to \lip_{c_4}(B_r, I), \quad 
    \left(\varphi, (s_1, s_2, \dots, s_N)\right) \mapsto  \varphi^\prime, \]
that satisfy the following conditions.
   \begin{enumerate}
      \item $|\varphi^\prime(t)-\varphi(t)| < \varepsilon$ for all $t\in B_r$.
      \item $\varphi^\prime(t) = \varphi(t)$ for all $t\in \partial B_r$.
      \item $L\left(\varphi^\prime, 0, \frac{\delta}{2}\right) = c_4$.
      \item $L\left(\varphi^\prime, t, \frac{\delta}{2}\right) \leq c_3$ for all $\delta \leq |t| \leq r-\frac{\delta}{2}$.
      \item For any $1\leq k \leq N$ we have $L(\varphi^\prime, p_k, \delta) = (1-s_k)c_2+ s_k c_3$. Moreover, for all 
               $t\in B_\delta(p_k)$ we have $L(\varphi^\prime, t, \delta) \leq L(\varphi^\prime, p_k, \delta)$.
               (Notice that we have $B_\delta(t) \cap B_\delta = \emptyset$ and
               $B_\delta(t) \cap B_\delta(p_\ell) = \emptyset$ for $\ell \neq k$ and $t \in B_\delta(p_k)$
               by (\ref{eq: choice of delta in multi-bumpy function}).)
   \end{enumerate}
\end{lemma}
Notice that the condition (\ref{eq: choice of delta in multi-bumpy function}) implies that the balls 
$B_\delta, B_\delta(p_1), B_\delta(p_2),\dots, B_\delta(p_N)$ are mutually disjoint and that they are all contained in 
$B_{\frac{r}{2}}$.
 
Probably the conditions (3), (4), (5) look unmotivated. 
Readers may imagine that $\varphi^\prime$ is a “multi-bumpy function” that has local peaks at $0, p_1, \dots, p_N$ and 
that the height of the peak at each $p_k$ is controlled by the parameter $s_k$.
In particular, we have “encoded” the value of $(s_1, \dots, s_N)$ to the function $\varphi^\prime$.
(This explanation is not precise but, hopefully, it will help readers to visualize the function $\varphi^\prime$.)  

\begin{proof}
We can assume $0<\varepsilon < 1$.
Let $\varepsilon_1\in \left(0, \frac{\varepsilon}{2}\right)$ be a small number such that we can take
$\psi\in \lip_{c_2-c_1}(B_r, I)$ which satisfies $0\leq \psi(t) \leq \varepsilon_1$ $(\forall t\in B_r)$, 
$\psi(t) = \varepsilon_1$ for $|t|\leq \frac{r}{2}$ and $\psi(t) = 0$ for $t\in \partial B_r$. 

Suppose we are given $\left(\varphi, (s_1, \dots, s_N)\right)\in \lip_{c_1}(B_r, I)\times I^N$.
Set $\varphi_1(t) = \max\{0, \varphi(t)-\psi(t)\}$.
Then $\varphi_1\in \lip_{c_2}(B_r, I)$, $|\varphi_1(t)-\varphi(t)|\leq \varepsilon_1 < \frac{\varepsilon}{2}$,
$\varphi_1 = \varphi$ on $\partial B_r$, and $0\leq \varphi_1(t)\leq 1-\varepsilon_1$ on $B_{\frac{r}{2}}$.
We take a positive number $\delta$ that satisfies (\ref{eq: choice of delta in multi-bumpy function}) and 
$2c_4 \delta < \varepsilon_1$.
We are going to define $\varphi^\prime = G\left(\varphi, (s_1, \dots, s_N)\right)$.

First we set $\varphi^\prime(t) = \varphi_1(t)$ for $t\in \{\frac{\delta}{2}\leq |t|\leq r\}\cap \bigcap_{k=1}^N\{|t-p_k|\geq \delta\}$.
For $t\in B_{\frac{\delta}{2}}$ we define 
\[ \varphi^\prime(t) = \inf\left\{\varphi_1(u) + c_4|t-u|\middle|\, u\in \partial B_{\frac{\delta}{2}}\right\}. \]
For $t\in B_\delta(p_k)$ $(1\leq k \leq N)$ we define (recall that $B_\delta(p_k) \subset B_r\setminus B_{\frac{\delta}{2}}$)
\[ \varphi^\prime(t) = \inf\left\{\varphi_1(u) + \left((1-s_k)c_2+s_k c_3\right)|t-u|\middle|\, u\in \partial B_{\delta}(p_k)\right\}. \]
Since $c_4\delta < \varepsilon_1$ and $\varphi_1(t) \leq 1-\varepsilon_1$ on $B_{\frac{r}{2}}$, 
we have $\varphi^\prime(t) \leq 1$ and 
\[ L\left(\varphi^\prime, 0, \frac{\delta}{2}\right) = c_4,\quad  L\left(\varphi^\prime, p_k, \delta\right) = (1-s_k)c_2 + s_k c_3. \]
We have defined $\varphi^\prime\in \lip_{c_4}(B_r, I)$ that satisfies the conditions (2) and (3).
It is also easy to see that $\varphi^\prime$ satisfies the conditions (4) and (5).

We claim that $|\varphi^\prime(t) - \varphi_1(t)| \leq \varepsilon_1 < \frac{\varepsilon}{2}$.
(This implies the condition (1).)
For $t\in B_{\frac{\delta}{2}}$ we take $u\in \partial B_{\frac{\delta}{2}}$ with $|t-u|\leq \frac{\delta}{2}$ and then
\begin{align*}
  \left|\varphi^\prime(t)-\varphi_1(t)\right| & \leq \left|\varphi^\prime(t)-\varphi^\prime(u)\right|
    + \left|\varphi_1(u) - \varphi_1(t)\right|, \quad \left(\text{by } \varphi^\prime(u) = \varphi_1(u)\right) \\
    & \leq c_4 \frac{\delta}{2} + c_4 \frac{\delta}{2} < \varepsilon_1, \quad (\text{by } c_4\delta < \varepsilon_1).
\end{align*}   
For $t\in B_\delta(p_k)$ we take $u\in \partial B_{\delta}(p_k)$ with $|t-u|\leq \delta$ and then 
\begin{align*}
  \left|\varphi^\prime(t)-\varphi_1(t)\right| & \leq \left|\varphi^\prime(t)-\varphi^\prime(u)\right|
    + \left|\varphi_1(u) - \varphi_1(t)\right|, \quad \left(\text{by } \varphi^\prime(u) = \varphi_1(u)\right) \\
    & \leq c_4 \delta + c_4 \delta < \varepsilon_1, \quad (\text{by } 2c_4\delta < \varepsilon_1).
\end{align*}   
\end{proof}

The next proposition will provide us a local model for constructing a perturbation map in the proof of 
Proposition \ref{proposition: C_k(B, B^prime) is open and dense}. Note that although Proposition \ref{prop: local perturbation map} resembles \cite[Lemma 4.7(1)(2)(4)]{GH}, its proof is rather different which may be useful in future generalizations.

\begin{proposition} \label{prop: local perturbation map}
Let $(K, d)$ be a compact metric space. Let $r>0$, $\varepsilon>0$ and $0<c<c^\prime$.
For any continuous map $f\colon K\to \lip_c\left(B_r(\mathbb{R}^n), I\right)$ there exists a continuous map 
$g\colon K\to \lip_{c^\prime}\left(B_r(\mathbb{R}^n), I\right)$ that satisfies the following three conditions.
  \begin{enumerate}
     \item $\left|f(x)(t)-g(x)(t)\right| < \varepsilon$ for all $x\in K$ and $t\in B_r$.
     \item $g(x)(t) = f(x)(t)$ for all $x\in K$ and $t\in \partial B_r$.
     \item Let $x, y\in K$. If there exists $s\in B_{\frac{r}{2}}$ satisfying 
     \[  \forall t\in B_{\frac{r}{2}}: \quad g(x)(t+s) = g(y)(t) \]
     then we have $d(x, y) < \varepsilon$.
  \end{enumerate}
\end{proposition}

\begin{proof}
Let $K=K_1\cup K_2\cup \dots \cup K_N$ be a covering of $K$ by compact sets $K_k$ with $\diam\, K_k < \varepsilon$ 
$(1\leq k \leq N)$.
We take open sets $U_k$ that contain $K_k$ and satisfy $\diam\, U_k < \varepsilon$.
We choose a continuous function $\chi_k \colon K\to I$ $(1\leq k \leq N)$ for which we have 
$\chi_k = 1$ on $K_k$ and $\supp \,\chi_k \subset U_k$.
We set $c_1 = c$, $c_2 =\frac{3c+c^\prime}{4}$, $c_3 = \frac{c+c^\prime}{2}$, $c_4 = c^\prime$.
We have $0<c_1<c_2<c_3<c_4$.

We take distinct points $p_1, \dots, p_N\in B_{\frac{r}{4}}\setminus \{0\}$. (Any choice will work.)
By using Lemma \ref{lemma: multi bump function}, we can take a positive number 
\[ \delta < \min\left\{\frac{r}{4}, \min_{k} \frac{|p_k|}{3}, \min_{k \neq \ell} \frac{|p_k-p_\ell|}{3}\right\} \]
and a continuous map 
\[ G\colon \lip_{c_1}(B_r, I)\times I^N \to \lip_{c_4}(B_r, I), \quad 
    \left(\varphi, (s_1, s_2, \dots, s_N)\right) \mapsto  \varphi^\prime, \]
that satisfy the following conditions.
\begin{enumerate}
      \item[(i)] $|\varphi^\prime(t)-\varphi(t)| < \varepsilon$ for all $t\in B_r$.
      \item[(ii)] $\varphi^\prime(t) = \varphi(t)$ for all $t\in \partial B_r$.
      \item[(iii)] $L\left(\varphi^\prime, 0, \frac{\delta}{2}\right) = c_4$.
      \item[(iv)] $L\left(\varphi^\prime, t, \frac{\delta}{2}\right) \leq c_3$ for all $\delta \leq |t| \leq r-\frac{\delta}{2}$.
      \item[(v)] For any $1\leq k \leq N$ we have $L(\varphi^\prime, p_k, \delta) = (1-s_k)c_2+ s_k c_3$. Moreover, for all 
               $t\in B_\delta(p_k)$ we have $L(\varphi^\prime, t, \delta) \leq L(\varphi^\prime, p_k, \delta)$.
   \end{enumerate}

We define a continuous map $g\colon K\to \lip_{c^\prime}(B_r, I)$ by 
setting 
\[ g(x) = G\left(f(x), (\chi_1(x), \dots, \chi_N(x))\right). \]
It follows from (i) and (ii) that, for any $x\in K$, we have $|f(x)(t)-g(x)(t)| < \varepsilon$ for $t\in B_r$ and 
$f(x)(t) = g(x)(t)$ for $t\in \partial B_r$.

Suppose that for some $x, y\in K$ and $s\in B_{\frac{r}{2}}$ we have $g(x)(s+t) = g(y)(t)$ for all $t\in B_{\frac{r}{2}}$.
Then 
\[ L\left(g(x), s, \frac{\delta}{2}\right) = L\left(g(y), 0, \frac{\delta}{2}\right) = c_4, \quad (\text{by } (\mathrm{iii})). \]
If $|s|\geq \delta$ then we have $L\left(g(x), s, \frac{\delta}{2}\right) \leq c_3 < c_4$ by (iv).
Therefore we must have $|s| < \delta$. 

We take $K_k$ that contains the point $y$.
Then $\chi_k(y) = 1$ and hence
\[ L\left(g(x), s+p_k, \delta\right) = L\left(g(y), p_k, \delta\right) = c_3, \quad (\text{by } (\mathrm{v})). \] 
It follows from (v) and $s+p_k \in B_\delta(p_k)$ that 
\[  c_3 = L\left(g(x), s+p_k, \delta\right) \leq \left(1-\chi_k(x)\right) c_2 + \chi_k(x) c_3. \]
Then we must have $\chi_k (x) = 1$. This implies $x\in U_k$ and hence $d(x, y) \leq \diam U_k < \varepsilon$.
\end{proof}

\section{Local section} \label{section: local section}

For $0\leq k \leq n$, we denote by $\gr(n, k)$ the Grassmannian manifold consisting of $k$-dimensional linear subspaces of 
$\mathbb{R}^n$.

Let $X$ be a compact metrizable space and $T\colon \mathbb{R}^n\times X\to X$ a continuous action of $\mathbb{R}^n$ on it.
For $x\in X$ we define $\mathfrak{g}_x$ as the connected component of the stabilizer of $x$, $\{t\in \mathbb{R}^n\mid T^t x = x\}$, containing the origin. As the stabilizer of $x$ is a closed subgroup of $\mathbb{R}^n$, we may use the well-known characterization of  closed subgroups of $\mathbb{R}^n$,\footnote{Let $H \subseteq \mathbb{R}^n$ be a closed subgroup. 
Then there exist integers $k,m \geq 0$ with $k+m \leq n$, and a linear embedding
$
  \varphi : \mathbb{R}^k \times \mathbb{Z}^m \;\longrightarrow\; \mathbb{R}^n
$
such that
$
  H = \varphi\!\left( \mathbb{R}^k \times \mathbb{Z}^m \right).
 $ } to conclude $\mathfrak{g}_x$ is the largest linear subspace $\mathfrak{g}\subset \mathbb{R}^n$
for which we have $T^t x =x$ ($\forall t\in \mathfrak{g}$).
We set $X_k = \{x\in X\mid \dim \mathfrak{g}_x\geq k\}$ for $0\leq k \leq n$ and $X_{n+1} = \emptyset$.
These are $T$-invariant closed subsets of $X$, and 
we have $X=X_0\supset X_1\supset X_2\supset \dots \supset X_n$.
For $x\in X$, let $\mathfrak{h}_x$ be the orthogonal complement of $\mathfrak{g}_x$ in $\mathbb{R}^n$.
Both $\mathfrak{g}_x$ and $\mathfrak{h}_x$ are $T$-invariant: $\mathfrak{g}_{T^t x} = \mathfrak{g}_x$ and 
$\mathfrak{h}_{T^t x} = \mathfrak{h}_x$ for any $t\in \mathbb{R}^n$.
The maps 
\[ X_k\setminus X_{k+1}\ni x\mapsto \mathfrak{g}_x\in \gr(n, k), \quad 
    X_k\setminus X_{k+1}\ni x \mapsto \mathfrak{h}_x \in \gr(n, n-k) \]
are both continuous. 

\begin{lemma} \label{lemma: local trivialization}
Let $p\in X_k\setminus X_{k+1}$ with $0\leq k \leq n$.
For a sufficiently small open neighborhood $U$ of $p$ in $X_k\setminus X_{k+1}$, we can find a linear isometry 
$\rho_x\colon \mathbb{R}^{n-k} \to \mathfrak{h}_x$ for each $x\in U$ such that $\rho_x$ is continuous in $x$ and $T$-invariant 
(i.e. if $x, T^u x \in U$ for some $u\in \mathbb{R}^n$ then $\rho_{T^u x} = \rho_x$).
\end{lemma}

A similar statement holds for $\mathfrak{g}_x$ as well. However we do not need it.

\begin{proof}
Let $\pi_x\colon \mathbb{R}^n\to \mathfrak{h}_x$ be the orthogonal projection.
Then $\pi_x$ is continuous in $x\in X_k\setminus X_{k+1}$ and $T$-invariant (i.e. $\pi_{T^t x} = \pi_x$).

We choose an orthonormal basis $e_1, e_2, \dots, e_{n-k}$ of $\mathfrak{h}_p$.
For a sufficiently small neighborhood $U$ of $p$ in $X_k\setminus X_{k+1}$, the vectors 
$\pi_x(e_1), \pi_x(e_2), \dots, \pi_x(e_{n-k})$ form a (not necessarily orthogonal) basis of $\mathfrak{h}_x$ for $x\in U$. 
We apply the Gram--Schmidt algorithm to them and obtain an orthonormal basis $u_{1, x}, \dots, u_{n-k, x}$ of $\mathfrak{h}_x$.
These vectors are continuous in $x\in U$ and $T$-invariant.
Therefore they provide a linear isometry $\rho_x\colon \mathbb{R}^{n-k}\to \mathfrak{h}_x$ that is continuous in $x\in U$ and 
$T$-invariant.
\end{proof}

We define $\mathfrak{h} = \{(x, t)\in X\times \mathbb{R}^n\mid t\in \mathfrak{h}_x\}$.
For a subset $A\subset X$, we denote $\mathfrak{h}|_A = \{(x, t)\in A\times \mathbb{R}^n\mid t\in \mathfrak{h}_x\}$.
For $r>0$ we also denote $B_r(\mathfrak{h}|_A) = \{(x, t)\in \mathfrak{h}|_A\mid |t|\leq r\}$ and 
$S_r(\mathfrak{h}|_A) = \{(x, t)\in \mathfrak{h}|_A\mid |t|=r\}$.

The following proposition is a main result of this subsection.

\begin{proposition} \label{prop: local section}
Let $\varepsilon>0$ and $p\in X_k\setminus X_{k+1}$ $(0\leq k \leq n-1)$.
Let $U$ be an open neighborhood of $p$ in $X_k\setminus X_{k+1}$. 
Then there exist a positive number $\delta<\varepsilon$ and a closed set $E$ of $X_k$ with $p\in E \subset U$ such that 
the map 
\[  B_\delta(\mathfrak{h}|_E) \to X_k, \quad (x, t) \mapsto T^t x \]
is injective and that its image contains an open neighborhood of $p$ in $X_k$.
\end{proposition}

\begin{proof}
We can find $0<r<\varepsilon$ with $p\not \in \{T^u p\mid u\in S_r(\mathfrak{h}_p)\}$.
Take a continuous function $h\colon X\to [0,1]$ such that $h=1$ on a neighborhood of $p$ and 
$\supp(h) \cap \{T^t p\mid t\in S_r(\mathfrak{h}_p)\} = \emptyset$.

We take $a\in (0, r)$ and a neighborhood $V$ of $p$ with $V\subset U$ for which we have
   \begin{enumerate}
    \item $h(T^s x) = 1$ for $(x, s)\in B_a\left(\mathfrak{h}|_V\right)$,
    \item $T^t x\not\in \supp(h)$ for any $(x, t)\in \mathfrak{h}|_V$ with $r-a \leq |t| \leq r+a$, 
    \item we can find a linear isometry $\rho_x\colon \mathbb{R}^{n-k}\to \mathfrak{h}_x$ for $x\in V$ such that 
             $\rho_x$ is continuous in $x\in V$ and $T$-invariant (i.e. $\rho_{T^u x} = \rho_x$ if $x, T^u x\in V$).    
   \end{enumerate}

In what follows, we denote $B_r = B_r(\mathbb{R}^{n-k}) = \{t\in \mathbb{R}^{n-k}\mid |t|\leq r\}$.
For $1\leq i \leq n-k$ and $c\in \mathbb{R}$ we also denote 
$H_i(c) = \{(t_1, \dots, t_{n-k})\in \mathbb{R}^{n-k}\mid t_i \leq c\}$ (a half space of $\mathbb{R}^{n-k}$).

We define $f = (f_1, \dots, f_{n-k})\colon V\to \mathbb{R}^{n-k}$ by 
\[ f_i(x) = \int_{B_r\cap H_i(0)} h\left(T^{\rho_x(t)}x\right)\, dt, \quad (1\leq i \leq n-k). \]
Here $dt$ denotes the standard Lebesgue measure on $\mathbb{R}^{n-k}$.

\begin{claim} \label{claim: f_i depends only on u_i}
Let $x\in V$ and $u = (u_1, \dots, u_{n-k}) \in B_a\left(\mathbb{R}^{n-k}\right)$ with $T^{\rho_x(u)}x\in V$. Then
\[ f_i\left(T^{\rho_x(u)}x\right) = \int_{B_r\cap H_i(u_i)} h\left(T^{\rho_x(t)}x\right)\, dt. \]
\end{claim}

\begin{proof}
Since $\rho_{T^{\rho_x(u)}x} = \rho_x$, we have
\begin{align*}
   f_i(T^{\rho_x(u)}x) &= \int_{B_r\cap H_i(0)} h\left(T^{\rho_x(t)}\cdot T^{\rho_x(u)}x\right)\, dt \\
    & = \int_{B_r\cap H_i(0)} h\left(T^{\rho_x(t+u)}x\right)\, dt \\
    & = \int_{u+B_r\cap H_i(0)} h\left(T^{\rho_x(t)}x\right)\, dt. 
\end{align*}    
We notice that the symmetric difference between $u+B_r\cap H_i(0)$ and $B_r\cap H_i(u_i)$ is contained in 
$\{t\in \mathbb{R}^{n-k}\mid r-a \leq |t| \leq r+a\}$.
Indeed, if $t\in \left(u+B_r\cap H_i(0)\right) \setminus \left(B_r\cap H_i(u_i)\right)$, then 
$|t|\leq |u| + r \leq a+r$ and $t_i \leq u_i$. 
We also have $|t|>r$ since $t\not \in B_r\cap H_i(u_i)$.
On the other hand, if $t\in \left(B_r\cap H_i(u_i)\right)\setminus \left(u+B_r\cap H_i(0)\right)$ then 
$|t|\leq r$ and $|t|>r-a$ (otherwise, $|t-u|\leq r$ and hence $t-u\in B_r\cap H_i(0)$).

If $r-a\leq |t|\leq r+a$ then $r-a\leq |\rho_x(t)|\leq r+a$ and hence $h\left(T^{\rho_x(t)}x\right) = 0$
by the condition (2).
Therefore 
\[ \int_{u+B_r\cap H_i(0)} h\left(T^{\rho_x(t)}x\right)\, dt = \int_{B_r\cap H_i(u_i)} h\left(T^{\rho_x(t)}x\right)\, dt. \]
\end{proof}

We choose $\delta \in \left(0, \frac{a}{2}\right)$ and a closed neighborhood $A$ of $p$ in $X_k$ such that $A\subset V$ 
and that $T^t x\in V$ for all $(x, t) \in B_{2\delta}\left(\mathfrak{h}|_A\right)$.
Then by Claim \ref{claim: f_i depends only on u_i}, for any $x\in A$ and $u\in B_{2\delta}(\mathbb{R}^{n-k})$
\[ f_i\left(T^{\rho_x(u)}x\right) = \int_{B_r\cap H_i(u_i)} h\left(T^{\rho_x(t)}x\right)\, dt. \]
This formula implies that, for each fixed $x\in A$, the value of the function $f_i\left(T^{\rho_x(u)}x\right)$
(over the domain $|u|\leq 2\delta$) depends only on the coordinate $u_i$.
Moreover, since $h(T^{\rho_x(t)}x) = 1$ on $t \in B_a\cap H_i(u_i)$, the function 
$f_i\left(T^{\rho_x(u)}x\right)$ is strictly monotone increasing in $u_i$.

We define $E = \{x\in A\mid f(x) = f(p)\}$. This is a closed subset of $X_k$ with $p\in E$.
We consider the map 
\begin{equation} \label{eq: flow box in the proof of local section}
 B_{\delta}\left(\mathfrak{h}|_E\right) \to X_k, \quad (x, t) \mapsto T^t x.  
\end{equation}
This map is injective. Indeed, suppose that $(x, t), (x^\prime, t^\prime)\in B_\delta(\mathfrak{h}|_E)$ satisfy 
$T^t x = T^{t^\prime} x^\prime$. 
Notice that this implies $\mathfrak{h}_x = \mathfrak{h}_{T^t x} = \mathfrak{h}_{T^{t^\prime} x^\prime} = 
\mathfrak{h}_{x^\prime}$.
We have $x^\prime = T^{t-t^\prime}x$.
Set $t-t^\prime = \rho_x(u)$ $(u\in \mathbb{R}^{n-k})$. We have $|u| = |t-t^\prime| \leq 2\delta$.
We want to show $u=0$. If not, we have $u_i\neq 0$ for some $1\leq i \leq n-k$.
If $u_i>0$ then (by the monotonicity of $f_i\left(T^{\rho_x(u)}x\right)$ in $u_i$)
\[ f_i(p) = f_i(x^\prime) = f_i\left(T^{\rho_x(u)}x\right) > f_i(x)  = f_i(p). \]
This is a contradiction. If $u_i < 0$ then 
\[  f_i(p) = f_i(x^\prime) = f_i\left(T^{\rho_x(u)}x\right) < f_i(x)  = f_i(p). \]
This is also a contradiction. Therefore the map (\ref{eq: flow box in the proof of local section}) is injective.

Next we claim that the image of the map (\ref{eq: flow box in the proof of local section}) contains an open neighborhood of $p$.
Take $0<\delta^\prime < \frac{\delta}{\sqrt{n-k}}$.
Let $e_1, \dots, e_{n-k}$ be the standard basis of $\mathbb{R}^{n-k}$.
For every $1\leq i \leq n-k$
\[ f_i\left(T^{-\rho_p(\delta^\prime e_i)}p\right) < f_i(p) < f_i\left(T^{\rho_p(\delta^\prime e_i)}p\right). \]
We can take a small neighborhood $A^\prime$ of $p$ in $X_k$ such that $A^\prime\subset A$ and that for every $x\in A^\prime$
\[  f_i\left(T^{-\rho_x(\delta^\prime e_i)}x\right) < f_i(p) < f_i\left(T^{\rho_x(\delta^\prime e_i)}x\right). \]
It follows from the mean value theorem that there is $u_i\in (-\delta^\prime, \delta^\prime)$ (depending on $x\in A^\prime$)
satisfying 
\[  f_i\left(T^{\rho_x(u_i e_i)} x\right) = f_i(p). \]
Set $u = (u_1, \dots, u_{n-k}) \in B_\delta(\mathbb{R}^{n-k})$. 
We have 
$f_i\left(T^{\rho_x(u)}x\right) = f_i(p)$ for all $1\leq i \leq n-k$,
and hence $f\left(T^{\rho_x(u)}x\right) = f(p)$.
We can assume $T^{\rho_x(u)}x \in A$ by choosing $\delta^\prime$ and $A^\prime$ sufficiently small.
Then $T^{\rho_x(u)}x \in E$.
Set $y = T^{\rho_x(u)} x$. We have $\rho_y = \rho_x$.
Therefore $x = T^{-\rho_y(u)}y$ is contained in the image of the map (\ref{eq: flow box in the proof of local section}).
This implies that the image of the map (\ref{eq: flow box in the proof of local section}) contains $A^\prime$.
\end{proof}

\section{Proofs of Propositions \ref{proposition: C_k(A) is open and dense} and \ref{proposition: C_k(B, B^prime) is open and dense}}
\label{section: proof of main propositions}

Here we prove Propositions \ref{proposition: C_k(A) is open and dense} and \ref{proposition: C_k(B, B^prime) is open and dense}
and complete the proof of Theorem \ref{theorem: main theorem}.

Let $X$ be a compact metrizable space and $T\colon \mathbb{R}^n\times X\to X$ a continuous action of 
$\mathbb{R}^n$ on it.
We use the notations ($\mathfrak{g}, \mathfrak{h}, X_k$ etc.) introduced in \S \ref{section: local section}.
Recall that $X_n = \fix(X, T) = \{x\in X\mid T^t x = x \, (\forall t\in \mathbb{R}^n)\}$.

Let $I=[0,1]$ be the unit interval, and let $Y:= \lip_1(\mathbb{R}^n, I)$ be the space of one-Lipschitz maps from 
$\mathbb{R}^n$ to $I$. We define a metric $\mathbf{d}$ on it by 
\[ \mathbf{d}(\varphi, \psi) = \sup_{m \geq 1} \left(2^{-m} \sup_{|t|\leq m} |\varphi(t)-\psi(t)|\right), \quad 
    \left(\varphi, \psi\in \lip_1(\mathbb{R}^n, I)\right). \]
$\mathbb{R}^n$ continuously acts on $Y$ by the shift:
\[ \sigma\colon \mathbb{R}^n\times Y\to Y, \quad  \left(u, \varphi(t)\right) \mapsto \varphi(t+u). \]
For $0\leq k \leq n$, we define $Y_k$ as the space of $\varphi\in Y$ such that there exists a $k$-dimensional linear subspace
$V\subset \mathbb{R}^n$ for which we have $\varphi(t+u) = \varphi(t)$ for all $t\in \mathbb{R}^n$ and $u\in V$.
We also set $Y_{n+1} = \emptyset$.
The space $Y_n$ consists of constant maps from $\mathbb{R}^n$ to $I$, 
and it can be identified with $I$.

\textbf{Throughout this section we assume that there exists a topological embedding 
$\iota \colon X_n\to Y_n$.}

We define $\mathcal{C}$ as the space of $\mathbb{R}^n$-equivariant continuous maps $f\colon X\to Y$ that satisfy
$f(x) = \iota(x)$ for all $x\in X_n$.
We define a metric $D$ on it by 
\[ D(f, g) = \max_{x\in X} \mathbf{d}(f(x), g(x)), \quad (f, g\in \mathcal{C}). \]
The space $\mathcal{C}$ becomes a complete metric space with respect to $D$.

\begin{lemma} \label{lemma: C is not empty}
The space $\mathcal{C}$ is not empty. 
Moreover, for any $\delta>0$ there exists $f\in \mathcal{C}$ such that 
$f(x) \in \lip_\delta(\mathbb{R}^n, I)$ for all $x\in X$.
\end{lemma}

\begin{proof}
We define $\iota_0\colon X_n\to I$ by $\iota_0(x) = \iota(x)(0)$.
By the Tietze extension theorem, it can be extended to a continuous function 
$\iota_0 \colon X\to I$.
We define an equivariant continuous map $\iota \colon X\to C(\mathbb{R}^n, I)$ by 
$\iota(x)(t) = \iota_0(T^t x)$ $(t\in \mathbb{R}^n)$, which is an extension of $\iota \colon X_n\to Y_n$.

We can assume $0<\delta < 1$.
By Proposition \ref{proposition: Lipschitz filter}, there is an equivariant continuous map 
$F\colon C(\mathbb{R}^n, I) \to  \lip_{\delta}(\mathbb{R}^n, I)$ such that 
$F(\varphi) = \varphi$ for all constant functions $\varphi \in C(\mathbb{R}^n, I)$.
We define $f\colon X\to \lip_\delta(\mathbb{R}^n, I)$ by 
$f(x) = F(\iota(x))$ for $x\in X$.
For $x\in X_n$, the function $\iota(x)$ is constant and hence $f(x) = F(\iota(x))  = \iota(x)$.
So $f\in \mathcal{C}$.
\end{proof}

Let $0\leq k \leq n-1$.
For a closed subset $A$ of $X_k$ with $A\cap X_{k+1} =\emptyset$, we define 
\[ \mathcal{C}_k(A) = \{f\in \mathcal{C}\mid f(A) \cap Y_{k+1} = \emptyset\}. \]
This is an open subset of $\mathcal{C}$.

\begin{proposition}[$=$ Proposition \ref{proposition: C_k(A) is open and dense}]
Let $0\leq k \leq n-1$ and $p\in X_k\setminus X_{k+1}$.
There exists a closed neighborhood $A$ of $p$ in $X_k$ with $A\cap X_{k+1} = \emptyset$ such that 
$\mathcal{C}_k(A)$ is open and dense in $\mathcal{C}$.
\end{proposition}

\begin{proof}
We take an open neighborhood $U$ of $p$ in $X_k$ with $U\cap X_{k+1}= \emptyset$ such that 
we can choose a linear isometry $\rho_x\colon \mathbb{R}^{n-k}\to \mathfrak{h}_x$ for each $x\in U$ which is 
continuous in $x\in U$ and $T$-invariant (i.e. $\rho_x = \rho_{T^t x}$ if $x, T^t x\in U$).
By Proposition \ref{prop: local section} we can find $\delta>0$ and a closed subset $E$ of $X_k$ with $E\cap X_{k+1}=\emptyset$
such that 
\begin{itemize}
   \item the map 
   \begin{equation} \label{eq: flow box in the proof of main prop 1}
     B_\delta(\mathfrak{h}|_E) \to X_k, \quad (x, t) \mapsto T^t x
   \end{equation}
    is injective,
    \item the image of the map (\ref{eq: flow box in the proof of main prop 1}) is contained in $U$,
    \item the image of the map (\ref{eq: flow box in the proof of main prop 1}) contains an open neighborhood $V$ of $p$ in $X_k$.
\end{itemize}

For $x\in X_k$ we set $H(x) = \{t\in \mathfrak{h}_x\mid T^t x\in E\}$.
We have $|t-u|>\delta$ for any two distinct $t, u\in H(x)$.

We choose a closed subset $A_0$ of $X_k$ such that $p\in A_0\subset E\cap V$ and that a set $A$ defined by
\[ A := \{T^t x\mid x\in A_0, t\in B_\delta(\mathfrak{h}_x)\} \]
is a closed neighborhood of $p$ in $X_k$.
Here we have denoted $B_\delta(\mathfrak{h}_x) = \{t\in \mathfrak{h}_x\mid |t|\leq \delta\}$.
See Figure \ref{figure: local section}.
Since $A\subset U$, we have $A\cap X_{k+1}=\emptyset$.
It is immediate to see that $\mathcal{C}_k(A)$ is open in $\mathcal{C}$.
We are going to prove that it is dense in $\mathcal{C}$.

\begin{figure}[h] 
    \centering
    \includegraphics[width=3.0in]{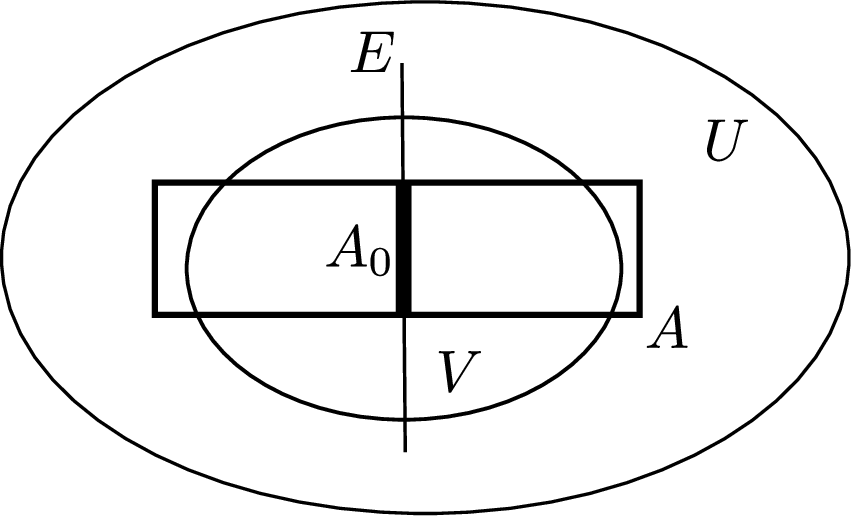}
    \caption{The schematic picture of the local section. The line segment in the center represents the local section $E$.
    The thick line segment inside $E$ represents the set $A_0$. The set $A_0$ contains the point $p$. 
    The rectangle represents the “flow box” $A$.}
     \label{figure: local section}
\end{figure}

We take a continuous function $\chi\colon E\to [0,1]$ such that $\chi(t) = 1$ for all $t\in A_0$ and 
$\supp(\chi) \subset E\cap V$.

Suppose we are given arbitrary $f\in \mathcal{C}$ and $0<\varepsilon<1$.
By Lemma \ref{lemma: C is not empty}, there is an equivariant continuous map 
$f_0\colon X\to \lip_{\frac{1}{2}}(\mathbb{R}^n, I)$ with $f_0(x) = \iota(x)$ for all $x\in X_n$.
We define an equivariant continuous map $f_1\colon X\to \lip_{1-\frac{\varepsilon}{2}}(\mathbb{R}^n, I)$ by 
$f_1(x)(t) = (1-\varepsilon)f(x)(t) + \varepsilon f_0(x)(t)$ for $x\in X$ and $t\in \mathbb{R}^n$.
We have $f_1\in \mathcal{C}$ and $|f_1(x)(t)-f(x)(t)| \leq 2 \varepsilon$.

By Proposition \ref{prop: local perturbation map, no translation invariance} we can take an equivariant continuous map 
\[ G\colon \lip_{1-\frac{\varepsilon}{2}}\left(B_{\frac{\delta}{2}}(\mathbb{R}^{n-k}), I\right) \to 
    \lip_{1-\frac{\varepsilon}{4}}\left(B_{\frac{\delta}{2}}(\mathbb{R}^{n-k}), I\right) \]
such that for all $\varphi\in  \lip_{1-\frac{\varepsilon}{2}}\left(B_{\frac{\delta}{2}}(\mathbb{R}^{n-k}), I\right)$
 \begin{itemize}
    \item $\left|G(\varphi)(t)-\varphi(t)\right| < \varepsilon$ for $t\in B_{\frac{\delta}{2}}(\mathbb{R}^{n-k})$,
    \item $G(\varphi)(t) = \varphi(t)$ for $t\in \partial B_{\frac{\delta}{2}}(\mathbb{R}^{n-k})$,
    \item $G(\varphi)$ has no translation invariance, that is, for any line $L\subset \mathbb{R}^{n-k}$ through the origin, 
    there are distinct $t, u\in B_{\frac{\delta}{2}}(\mathbb{R}^{n-k})$ with $t-u\in L$ and $G(\varphi)(t) \neq G(\varphi)(u)$.
 \end{itemize}

For each $x\in U$, we consider a map 
\[ f_1(x)\circ \rho_x\colon B_{\frac{\delta}{2}}(\mathbb{R}^{n-k})\to I, \quad t \mapsto f_1(x)\left(\rho_x(t)\right). \]
This map belongs to $\lip_{1-\frac{\varepsilon}{2}}\left(B_{\frac{\delta}{2}}(\mathbb{R}^{n-k}), I\right)$.
For $t\in B_{\frac{\delta}{2}}(\mathfrak{h}_x) = \{t\in \mathfrak{h}_x\mid |t|\leq \frac{\delta}{2}\}$,
we set $f_2(x)(t) = G\left(f_1(x)\circ \rho_x\right)(\rho_x^{-1}(t))$.
For any $x\in U$ we have $f_2(x)\in \lip_{1-\frac{\varepsilon}{4}}\left(B_{\frac{\delta}{2}}(\mathfrak{h}_x), I\right)$ and it satisfies
\begin{enumerate}
   \item[(i)] $\left|f_2(x)(t) - f_1(x)(t)\right| < \varepsilon$ for $t \in B_{\frac{\delta}{2}}(\mathfrak{h}_x)$,
   \item[(ii)] $f_2(x)(t) = f_1(x)(t)$ for $t\in \partial  B_{\frac{\delta}{2}}(\mathfrak{h}_x)$,
   \item[(iii)] $f_2(x)$ has no translation invariance, that is, for any line $L\subset \mathfrak{h}_x$ through the origin, there are two distinct 
                  points $t, u\in B_{\frac{\delta}{2}}(\mathfrak{h}_x)$ with $t-u\in L$ and $f_2(x)(t) \neq f_2(x)(u)$.
\end{enumerate}

Let $x\in X_k$. We define $g(x)\in \lip_{1-\frac{\varepsilon}{4}}(\mathbb{R}^n, I)$ in the following two steps.
  \begin{enumerate}
     \item[(a)] For $t\in \mathfrak{h}_x\setminus \bigcup_{s\in H(x)} \left(s + B_{\frac{\delta}{2}}(\mathfrak{h}_x)\right)$ and $u\in \mathfrak{g}_x$
                    we set $g(x)(t+u) = f_1(x)(t+u)$ (which is equal to $f_1(x)(t)$). In particular, if $H(x) =\emptyset $ then $g(x) = f_1(x)$.
     \item[(b)] Let $s\in H(x)$. For $t\in s + B_{\frac{\delta}{2}}(\mathfrak{h}_x)$ and $u\in \mathfrak{g}_x$ we define 
     \begin{align*}
       g(x)(t+u) & = g(x)(t) \\
                    & = \left(1-\chi(T^s x)\right) f_1(x)(t) + \chi(T^s x)f_2(T^s x)(t-s)  \\
                    & = f_1(x)(t) + \chi(T^s x)\left(f_2(T^s x)(t-s) - f_1(T^s x)(t-s)\right).
     \end{align*}  
     Here we have used $f_1(T^s x)(t-s) = f_1(x)(t)$. Notice that we have $g(x)(t+u) = f_1(x)(t)$ 
     for $t\in s + \partial  B_{\frac{\delta}{2}}(\mathfrak{h}_x)$ and $u\in \mathfrak{g}_x$ by the above condition (ii).
     Therefore these two steps (a) and (b) are compatible.
  \end{enumerate}

We have $\left|g(x)(t) - f_1(x)(t)\right| < \varepsilon$ for all $t\in \mathbb{R}^n$.
If $x\in A$, then there is $s\in H(x)$ with $T^s x\in A_0$ ($\Rightarrow \chi(T^s x) = 1$).
We have $g(x)(t) = f_2(T^s x)(t-s)$ for $t\in s + B_{\frac{\delta}{2}}(\mathfrak{h}_x)$.
Then it follows from (iii) that $g(x)\not \in Y_{k+1}$.

For $x\in X_{k+1}$ we have $H(x) = \emptyset$ because $X_{k+1}\cap U = \emptyset$.
Then $g(x) = f_1(x)$. In particular we have $g(x) = \iota(x)$ for $x\in X_n$.

Therefore we have constructed an equivariant continuous map 
$g\colon X_k\to \lip_{1-\frac{\varepsilon}{4}}(\mathbb{R}^n, I)$ such that 
  \begin{itemize}
    \item $|g(x)(t) - f_1(x)(t)| < \varepsilon$ for all $x\in X_k$ and $t\in \mathbb{R}^n$, 
    \item $g(A)\cap Y_{k+1}=\emptyset$,
    \item $g(x) = \iota(x)$ for $x\in X_n$.  
  \end{itemize}
From the second condition we can take $\varepsilon^\prime>0$ for which we have 
$\mathbf{d}\left(g(x), \psi\right) > \varepsilon^\prime$ for any $x\in A$ and $\psi \in Y_{k+1}$.
(Recall that $\mathbf{d}$ is the metric on $Y$ defined by 
 $\mathbf{d}(\varphi, \psi) = \sup_{m \geq 1} \left(2^{-m}\sup_{|t|\leq m} |\varphi(t)-\psi(t)|\right)$.)

By Proposition \ref{prop: extension and filter}, we can find an equivariant continuous map 
$h\colon X \to \lip_1 (\mathbb{R}^n, I)$ that satisfies the following three conditions.
  \begin{itemize}
   \item $|h(x)(t) - f_1(x)(t)| < \varepsilon$ for $x\in X$ and $t\in \mathbb{R}^n$.
   \item $|h(x)(t) - g(x)(t)| < \varepsilon^\prime$ for $x\in X_k$ and $t\in \mathbb{R}^n$.
            This condition implies that $h(x)\not\in Y_{k+1}$ for all $x\in A$.
   \item $h(x) = g(x) = \iota(x)$ for $x\in X_n$. This implies that $h\in \mathcal{C}$.  
  \end{itemize} 
We have $h\in \mathcal{C}_k(A)$.
For $x\in X$ and $t\in \mathbb{R}^n$
\[ |h(x)(t)-f(x)(t)| \leq |h(x)(t)-f_1(x)(t)| + |f_1(x)(t) - f(x)(t)| < 3\varepsilon. \]
Since $f\in \mathcal{C}$ and $\varepsilon \in (0, 1)$ are arbitrary, this has shown that $\mathcal{C}_k(A)$ is dense in 
$\mathcal{C}$. 
\end{proof}

Let $0\leq k \leq n-1$.
We define $\tilde{X}_k$ as the set of $(x, y)\in (X_k\setminus X_{k+1})\times (X_k\setminus X_{k+1})$ that satisfies 
$x\neq y$ and $\mathfrak{g}_x = \mathfrak{g}_y$.

\begin{lemma} \label{lemma: local trivialization revisited}
Let $(p_1, p_2)\in \tilde{X}_k$. There is an open neighborhood $U_i$ of $p_i$ in $X_k \setminus X_{k+1}$ $(i=1,2)$ for which the following 
statement holds. We can take a linear isometry $\rho_x\colon \mathbb{R}^{n-k} \to \mathfrak{h}_x$ for each $x\in U_1\cup U_2$ 
such that $\rho_x$ is continuous in $x\in U_1\cup U_2$ and that $\rho_x = \rho_y$ for any $x, y\in U_1\cup U_2$ with 
$\mathfrak{g}_x = \mathfrak{g}_y$ (in particular, if $x, T^u x\in U_1\cup U_2$ for some $u \in \mathbb{R}^n$ then $\rho_x = \rho_{T^u x}$).
\end{lemma}

\begin{proof}
The proof is the same as in Lemma \ref{lemma: local trivialization}.
Let $\pi_x\colon \mathbb{R}^n\to \mathfrak{h}_x$ be the orthogonal projection.
Then $\pi_x$ is continuous in $x\in X_k\setminus X_{k+1}$ and 
$\pi_x = \pi_y$ for any $x, y\in X_k\setminus X_{k+1}$ with $\mathfrak{g}_x = \mathfrak{g}_y$.
We choose an orthonormal basis $e_1, e_2, \dots, e_{n-k}$ of $\mathfrak{h}_{p_1} = \mathfrak{h}_{p_2}$.
For a sufficiently small neighborhood $U_i$ of $p_i$ in $X_k\setminus X_{k+1}$ $(i=1,2)$, the vectors 
$\pi_x(e_1), \pi_x(e_2), \dots, \pi_x(e_{n-k})$ form a (not necessarily orthogonal) basis of $\mathfrak{h}_x$ for $x\in U_1\cup U_2$. 
We apply the Gram--Schmidt algorithm to them and obtain an orthonormal basis $u_{1, x}, \dots, u_{n-k, x}$ of $\mathfrak{h}_x$.
These vectors are continuous in $x\in U$, and $(u_{1,x}, \dots, u_{n-k,x}) = (u_{1,y}, \dots, u_{n-k, y})$ for  
$x, y\in U_1\cup U_2$ with $\mathfrak{g}_x = \mathfrak{g}_y$.
Therefore they provide a linear isometry $\rho_x\colon \mathbb{R}^{n-k}\to \mathfrak{h}_x$ that satisfies the requirement.
\end{proof}

For closed subsets $B_1, B_2$ of $X_k$ with $B_1\cap B_2 = B_1 \cap X_{k+1} = B_2 \cap X_{k+1} = \emptyset$,
we define $\mathcal{C}_k(B_1, B_2)$ as the set of $f\in \mathcal{C}$ that satisfies 
$f(x)\neq f(y)$ for all $(x, y)\in (B_1 \times B_2) \cap \tilde{X}_k$.
Since $(B_1 \times B_2) \cap \tilde{X}_k$ is a closed subset of $X_k\times X_k$, the set 
$\mathcal{C}_k(B_1, B_2)$ is open in $\mathcal{C}$.

\begin{proposition}[$=$ Proposition \ref{proposition: C_k(B, B^prime) is open and dense}]
Let $0\leq k \leq n-1$ and $(p_1, p_2)\in \tilde{X}_k$.
There exist closed neighborhoods $B_i$ of $p_i$ in $X_k$ for $i=1,2$ such that
$B_1\cap B_2 = B_1 \cap X_{k+1} = B_2 \cap X_{k+1} = \emptyset$ and that 
$\mathcal{C}_k(B_1, B_2)$ is open and dense in $\mathcal{C}$.
\end{proposition}

\begin{proof}
By Lemma \ref{lemma: local trivialization revisited}, there is an open neighborhood $U_i$ of $p_i$ in $X_k\setminus X_{k+1}$ $(i=1,2)$ with 
$U_1\cap U_2 = \emptyset$ so that we can choose a linear isometry $\rho_x\colon \mathbb{R}^{n-k}\to \mathfrak{h}_x$ 
for $x\in U_1\cup U_2$ which is continuous in $x\in U_1\cup U_2$ and satisfies 
$\rho_x  = \rho_y$ for any $x, y\in U_1\cup U_2$ with $\mathfrak{g}_x = \mathfrak{g}_y$.

By Proposition \ref{prop: local section} we can take $\delta>0$ and a closed subset $E_i$ of $X_k$ with $p_i\in E_i \subset U_i$ and
$E_i\cap X_{k+1}=\emptyset$ $(i=1,2)$ such that 
\begin{itemize}
   \item the map 
   \begin{equation} \label{eq: flow box in the proof of main proposition 2}
     B_{\delta}(\mathfrak{h}|_{E_i}) \to X_k, \quad (x, t) \mapsto T^t x 
   \end{equation}
   is injective,
   \item the image of the map (\ref{eq: flow box in the proof of main proposition 2}) is contained in $U_i$, 
   \item the image of the map (\ref{eq: flow box in the proof of main proposition 2}) contains an open neighborhood $V_i$ of $p_i$ in $X_k$.
\end{itemize}

For $x\in X_k$ we define $H(x) = \{t\in \mathfrak{h}_x\mid T^t x \in E_1\cup E_2\}$.
We have $|t-u|>\delta$ for any $t, u\in H(x)$ with $t\neq u$.
We also have $H(x) = \emptyset$ for $x\in X_{k+1}$.

We can find a closed subset $A_i$ of $X_k$ such that $p_i\in A_i\subset E_i\cap V_i$ and that a set $B_i$ defined by
\[ B_i = \{T^t x\mid x\in A_i, t \in B_{\frac{\delta}{8}}(\mathfrak{h}_x)\} \]
is a closed neighborhood of $p_i$ in $X_k$.
See Figure \ref{figure: local section2}.

\begin{figure}[h] 
    \centering
    \includegraphics[width=3.0in]{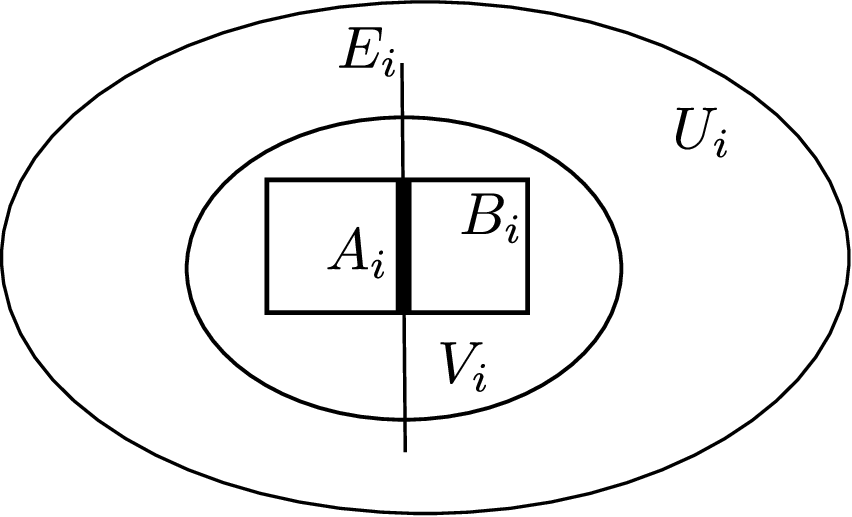}
    \caption{The schematic picture of the local section $E_i$ and the flow box $B_i$. 
    The line segment in the center represents the local section $E_i$, and the rectangle represents $B_i$. 
    The point $p_i$ belongs to the set $A_i$ (the thick line segment inside $E_i$).}
     \label{figure: local section2}
\end{figure}

Notice that $B_i\subset U_i$ and hence $B_1\cap B_2 = B_1\cap X_{k+1} = B_2\cap X_{k+1} = \emptyset$.
We are going to prove that $\mathcal{C}_k(B_1, B_2)$ is dense in $\mathcal{C}$.

We take a continuous function $\chi\colon E_1\cup E_2\to I$ such that 
$\chi(x) = 1$ for all $x\in A_1\cup A_2$ and $\supp(\chi) \subset (E_1\cap V_1) \cup (E_2\cap V_2)$.

Suppose we are given arbitrary $f\in \mathcal{C}$ and $0<\varepsilon<1$.
By Lemma \ref{lemma: C is not empty}, we can take an equivariant continuous map 
$f_0\colon X\to \lip_{\frac{1}{2}}(\mathbb{R}^n, I)$ such that $f_0(x) = \iota(x)$ for all $x\in X_n$.
We define $f_1\colon X\to \lip_{1-\frac{\varepsilon}{2}}(\mathbb{R}^n, I)$ by 
$f_1(x)(t) = (1-\varepsilon) f(x)(t) + \varepsilon f_0(x)(t)$. We have $f_1\in \mathcal{C}$ and 
$|f_1(x)(t)-f(x)(t)| \leq 2\varepsilon$ for all $x\in X$ and $t\in \mathbb{R}^n$. 

We take a metric $d$ on $X$. We denote $d(E_1, E_2)  := \min\{d(x,y)\mid x\in E_1, y\in E_2\}$.
This is positive because $E_1$ and $E_2$ are disjoint closed subsets.
Let $x\in E_1\cup E_2$ we consider the map 
\[ f_1(x)\circ \rho_x\colon B_{\frac{\delta}{2}}(\mathbb{R}^{n-k}) \to I, \quad t\mapsto f_1(x)\left(\rho_x(t)\right). \]
We apply Proposition \ref{prop: local perturbation map} to this map (with $K := E_1\cup E_2$).
Then we get a continuous map 
$g_1\colon E_1\cup E_2\to \lip_{1-\frac{\varepsilon}{4}}\left(B_{\frac{\delta}{2}}(\mathbb{R}^{n-k}), I\right)$
that satisfies the following three conditions.
 \begin{itemize}
    \item $|f_1(x)\left(\rho_x(t)\right) - g_1(x)(t)| < \varepsilon$ for all $x\in E_1\cup E_2$ and $t\in B_{\frac{\delta}{2}}(\mathbb{R}^{n-k})$,
    \item  $g_1(x)(t) = f_1(x)\left(\rho_x(t)\right)$ for all $x\in E_1\cup E_2$ and $t\in \partial B_{\frac{\delta}{2}}(\mathbb{R}^{n-k})$,
    \item  Let $x, y \in E_1\cup E_2$.
              If there is $s \in B_{\frac{\delta}{4}}(\mathbb{R}^{n-k})$ satisfying 
             $g_1(x)(t+s) = g_1(y)(t)$ for all $t \in B_{\frac{\delta}{4}}(\mathbb{R}^{n-k})$ then $d(x, y) < d(E_1, E_2)$
             (in particular, it is impossible to have both $x\in E_1$ and $y\in E_2$).
 \end{itemize}

For $x\in E_1\cup E_2$ we define $f_2(x)(t) = g_1(x)\left(\rho_x^{-1}(t)\right)$ for $t\in B_{\frac{\delta}{2}}(\mathfrak{h}_x)$.
We have $f_2(x) \in \lip_{1-\frac{\varepsilon}{4}}\left(B_{\frac{\delta}{2}}(\mathfrak{h}_x), I\right)$ and it satisfies 
the following three conditions.
   \begin{enumerate}
      \item[(i)] $|f_1(x)(t)- f_2(x)(t)| < \varepsilon$ for $x\in E_1\cup E_2$ and $t\in B_{\frac{\delta}{2}}(\mathfrak{h}_x)$.
      \item[(ii)] $f_2(x)(t) = f_1(x)(t)$ for $x\in E_1\cup E_2$ and $t\in \partial B_{\frac{\delta}{2}}(\mathfrak{h}_x)$.     
      \item[(iii)] Let $x, y\in E_1\cup E_2$ with $\mathfrak{g}_x = \mathfrak{g}_y$. 
                     If there is $s\in B_{\frac{\delta}{4}}(\mathfrak{h}_x)$
                     satisfying $f_2(x)(t+s) = f_2(y)(t)$ for all $t\in B_{\frac{\delta}{4}}(\mathfrak{h}_x)$ then 
                     $d(x, y) < d(E_1, E_2)$.
   \end{enumerate}

We define an equivariant continuous map $g\colon X_k\to \lip_{1-\frac{\varepsilon}{4}}(\mathbb{R}^n, I)$ as follows.
Let $x\in X_k$.
   \begin{enumerate}
      \item[(a)] For $t\in \mathfrak{h}_x \setminus \bigcup_{s\in H(x)} \left(s+ B_{\frac{\delta}{2}}(\mathfrak{h}_x)\right)$ and $u\in \mathfrak{g}_x$
               we set $g(x)(t+u) = f_1(x)(t)$. In particular we have $g(x) = f_1(x)$ if $H(x) =\emptyset$.
      \item[(b)] Let $s\in H(x)$. For $t\in s + B_{\frac{\delta}{2}}(\mathfrak{h}_x)$ and $u\in \mathfrak{g}_x$ we set
      \begin{align*}
         g(x)(t+u) & = g(x)(t) \\
                      & = \left(1-\chi(T^s x)\right)f_1(x)(t) + \chi(T^s x) f_2(T^s x)(t-s) \\
                      & = f_1(x)(t) + \chi(T^s x)\left(f_2(T^s x)(t-s) - f_1(T^s x)(t-s)\right).
      \end{align*}
       This is compatible with the above (a) by the condition (ii) of $f_2$.
   \end{enumerate}
We have $|g(x)(t) - f_1(x)(t)| < \varepsilon$ for all $x\in X_k$ and $t\in \mathbb{R}^n$.
We also have $g(x) = \iota(x)$ for $x\in X_n$ because $H(x) = \emptyset$ for $x\in X_n$.

\begin{claim}
If $x_1\in B_1$ and $x_2 \in B_2$ satisfy $\mathfrak{g}_{x_1} = \mathfrak{g}_{x_2}$ then $g(x_1)\neq g(x_2)$.
\end{claim}

\begin{proof}
Suppose $g(x_1) = g(x_2)$.
There is $s_i\in B_{\frac{\delta}{8}}(\mathfrak{h}_{x_i})$ with $T^{s_i} x_i \in A_i$ $(i=1,2)$.
Set $s := s_2-s_1\in B_{\frac{\delta}{4}}(\mathfrak{h}_{x_1})$.
Since $\chi = 1$ on $A_1\cup A_2$,
for $t\in B_{\frac{\delta}{2}}(\mathfrak{h}_{x_1})$ we have 
\[  g(T^{s_1}x_1)(t) = f_2(T^{s_1}x_1)(t), \quad g(T^{s_2}x_2)(t) = f_2(T^{s_2}x_2)(t). \] 
By the equivariance, we have $g(T^{s_1}x_1)(t+s) = g(T^{s_2} x_1)(t) = g(T^{s_2} x_2)(t)$. Then 
\[ f_2(T^{s_1}x_1)(t+s) = f_2(T^{s_2} x_2)(t) \quad (t\in B_{\frac{\delta}{4}}(\mathfrak{h}_{x_1})). \]
It follows from the condition (iii) of $f_2$ that we have $d(T^{s_1} x_1, T^{s_2} x_2) < d(E_1, E_2)$.
However this contradicts $T^{s_i}x_i \in A_i \subset E_i$.
\end{proof}

Then we can find a positive number $\varepsilon^\prime$ for which we have 
$\mathbf{d}(g(x), g(y)) > \varepsilon^\prime$ for all $(x, y)\in B_1\times B_2$ with $\mathfrak{g}_x = \mathfrak{g}_y$.
(Recall again that $\mathbf{d}$ is the metric on $Y$ defined by 
 $\mathbf{d}(\varphi, \psi) = \sup_{m \geq 1} \left(2^{-m}\sup_{|t|\leq m} |\varphi(t)-\psi(t)|\right)$.)
 
By Proposition \ref{prop: extension and filter}, we can take an equivariant continuous map 
$h\colon X\to \lip_1(\mathbb{R}^n, I)$ satisfying the following conditions.
  \begin{itemize}
     \item $|h(x)(t)-f_1(x)(t)| < \varepsilon$ for all $x\in X$ and $t\in \mathbb{R}^n$.
     \item $|h(x)(t)-g(x)(t)| < \varepsilon^\prime$ for all $x\in X_k$ and $t\in \mathbb{R}^n$. This implies that
              $h(x) \neq  h(y)$ for any $(x, y)\in B_1\times B_2$ with $\mathfrak{g}_x = \mathfrak{g}_y$.  
     \item $h(x) = g(x) = \iota(x)$ for all $x\in X_n$. 
  \end{itemize}
Hence $h\in \mathcal{C}_k(B_1, B_2)$ and we have 
\[  |h(x)(t) - f(x)(t)|  \leq  |h(x)(t)-f_1(x)(t)| + |f_1(x)(t)-f(x)(t)|  < 3\varepsilon \]
for all $x\in X$ and $t\in \mathbb{R}^n$.
Since $f\in \mathcal{C}$ and $\varepsilon \in (0,1)$ are arbitrary, this has shown that 
$\mathcal{C}_k(B_1, B_2)$ is dense in $\mathcal{C}$.
\end{proof}

As we have already explained in \S \ref{subsection: strategy of the proof}, Lemma \ref{lemma: C is not empty} and 
Propositions \ref{proposition: C_k(A) is open and dense} and \ref{proposition: C_k(B, B^prime) is open and dense}
imply our main theorem (Theorem \ref{theorem: main theorem}).
Therefore we have completed the proof of the main theorem.

\end{document}